\documentclass[11]{amsart}

\newtheorem{theorem}{Theorem}[section]
\newtheorem{lemma}[theorem]{Lemma}

\theoremstyle{definition}
\newtheorem{definition}[theorem]{Definition}

\newtheorem{proposition}[theorem]{Proposition}

\newtheorem{hypothesis}[theorem]{Hypothesis}

\theoremstyle{remark}
\newtheorem{remark}[theorem]{Remark}

\numberwithin{equation}{section}

\usepackage{graphicx, dsfont}

\usepackage{amsmath,subfig}
\usepackage{paralist}
\usepackage{graphics} 
\usepackage{epsfig} 
\usepackage{graphicx}  \usepackage{epstopdf}

\usepackage[]{algorithm2e}
\usepackage{hhline}

 \usepackage[colorlinks=true]{hyperref}
\hypersetup{urlcolor=blue, citecolor=red}

\usepackage{graphicx,amsmath}
\newcommand{\norm}[2]{\left\lVert#1\right\rVert_{#2}}

\newcommand{\RR}{\mathds{R}}
\newcommand{\NNN}{\mathds{N}}

\newcommand{\id}{{\rm id}}

\newcommand\finsquare{
\end{proof}\medskip } 

\numberwithin{equation}{section}

\newcommand{\Gammaint}{\Gamma_{\text{int}}}
\newcommand{\Gammain}{\Gamma_{\text{in}}}
\newcommand{\Gammaext}{\Gamma_{\text{ext}}}

\newcommand{\Gammaout}{\Gamma_{\text{out}}}
\newcommand{\Gammawall}{\Gamma_{\text{wall}}}
\newcommand{\Phibb}{\Phi}
\newcommand{\cof}{\operatorname{cof}}
\newcommand{\nx}{n_x}
\newcommand{\ny}{n_y}
\newcommand{\dddiv}{\operatorname{div}}

\newcommand{\Jh}{\hat{J}}

\usepackage{pgfplotstable}
\usepackage{pgfplots}
\usepackage{xcolor}
\usepackage{graphicx, dsfont}

\usepackage{amsmath,subfig}
\usepackage{paralist}
\usepackage{graphics} 
\usepackage{epsfig, booktabs,multirow} 
\usepackage{graphicx}  \usepackage{epstopdf}

\usepackage[]{algorithm2e}
\usepackage{hhline}

 \usepackage[colorlinks=true]{hyperref}

\newcommand{\R}{\mathds{R}}

\newcommand{\rred}{\textcolor{black}}
\newcommand{\cyan}{\textcolor{black}}
\newcommand{\brown}{\textcolor{black}}
\newcommand{\blue}{\textcolor{black}}
\newcommand{\purple}{\textcolor{black}}
\newcommand{\orange}{\textcolor{black}}
\newcommand{\green}{\textcolor{black}}

\newcommand{\Wp}{W^p_{\widehat{\Om}_2}}
\newcommand{\Wpk}{W^p_{\Om_{\kappa}}}

\newcommand{\Ww}{W_{w,\widehat{\Om}_2}}
\newcommand{\Wpp}{W_{p,\widehat{\Om}_2}}

\def\dd{{\rm d}}

\def\weight(#1,#2){c_{#1,#2}}

\if {

}{{\caption}}
} \fi

\def\fh{\hat{f}}
\def\gh{\hat{g}}

\def\ph{\hat{p}}

\def\uh{\hat{u}}
\def\vh{\hat{v}}
\def\wh{\hat{w}}

\def\pb{\bar{p}}

\def\ub{\bar{u}}

\def\wb{\bar{w}}

\def\zb{\bar{z}} 

\def\Yb{\bar{Y}}
\def\Zb{\bar{Z}}

\def\ft{\tilde{f}}

\def\pt{\tilde{p}}

\def\wt{\tilde{w}}

\def\zt{\tilde{z}}

\def\cala{{\mathcal  A}}

\def\cald{{\mathcal D}}

\def\calf{{\mathcal F}}
\def\calg{{\mathcal G}}
\def\calh{{\mathcal H}}

\def\caln{{\mathcal N}}

\def\cals{{\mathcal S}}

\def\Om{{\Omega}}

\def\psit{{\tilde\psi}}

\def\1B{{\bf  1}}

\def\det{\mathop{\rm det}}

\def\half{\mbox{$\frac{1}{2}$}}

\def\1B{{\bf  1}}

\newcommand\be{\begin{equation}}
\newcommand\ee{\end{equation}}
\newcommand\ba{\begin{array}}
\newcommand\ea{\end{array}}
\newcommand{\bea}{\begin{eqnarray}}
\newcommand{\eea}{\end{eqnarray}}
\newcommand{\bean}{\begin{eqnarray*}}
\newcommand{\eean}{\end{eqnarray*}}

\def\rar{\rightarrow}

\begin{document}

\title[Differentiability properties for boundary control of FSI problems]
{Differentiability properties for boundary control of fluid-structure interactions of linear elasticity with Navier-Stokes equations with mixed-boundary conditions in a channel}


\author{Michael Hinterm\"uller}
\address{Weierstrass Institute for Applied Analysis and Stochastics\\
            Mohrenstr. 39\\
            10117 Berlin, Germany}
\curraddr{}
\email{michael.hintermueller@wias-berlin.de}
\thanks{}

\author{Axel Kr\"oner}
\address{Weierstrass Institute
for Applied Analysis and Stochastics\\
            Mohrenstr. 39 \\
            10117 Berlin, Germany}
\curraddr{}
\email{axel.kroener@wias-berlin.de}

\subjclass[2020]{74F10}

\keywords{fluid-structure interaction, boundary control, differentiability properties,  Navier Stokes equation, mixed boundary conditions,  domain with corners.
}

\date{\today}

\dedicatory{}

\begin{abstract}
In this paper we consider a fluid-structure interaction problem given by the steady Navier Stokes equations coupled with linear elasticity taken from [\emph{Lasiecka, Szulc, and Zochoswki, Nonl. Anal.: Real World Appl., 44, 2018}]. 
An elastic body surrounded by a liquid in a rectangular domain is deformed by the flow which can be controlled by the Dirichlet boundary condition at the inlet. 
On the walls along the channel homogeneous Dirichlet boundary conditions and on the outflow boundary do-nothing conditions are prescribed.
We recall existence results for the nonlinear system from that reference and analyze the control to state mapping generalizing the results of [\emph{Wollner and Wick, J. Math. Fluid Mech.,
21, 2019}] to the setting of the nonlinear Navier-Stokes equation for the fluid and the situation of mixed boundary conditions in a domain with corners.
\end{abstract}

\maketitle

\tableofcontents

\section{Introduction}

The paper deals with fluid-structure interaction (FSI) problems given by a fluid flow around an elastic body in a rectangular channel with fixed walls in two space dimensions. The elastic body deforms under the flow and is modelled by linear elasticity, for the fluid we consider the steady Navier-Stokes equation with Dirichlet condition at the inlet, no-slip condition on the wall, and do-nothing condition on the outlet. The configuration is taken from Lasiecka, Szulc, and Zochoswki  \cite{MR3825153} who analyze existence of solutions to this FSI problem and existence of an optimal inflow profile, considered as a boundary control, which minimizes the drag at the interface of the elastic body and the fluid. 
  Let $g$ denote the Dirichlet inflow boundary values and  $(u,w,p)$ be the solution of the FSI problem after transforming the variables for the fluid to a reference domain, that means $u$ solves the elasticity equation, $(w,p)$ is the solution of the Navier-Stokes equation and both equations are coupled via the traction force at the interface and via coefficients in the Navier-Stokes equation. We show that the control to state map of the FSI problem
\be
B_r(\calg_{3/2}) \rar   X^p,\quad  g \mapsto (u,w,p)
\ee
with ball $B_r(\calg_{3/2})$ around zero with radius $r>0$ in the  space $\calg_{3/2}$ defined in \eqref{space-g} and $X^p$, $p>2$, defined in \eqref{Xp} is continuously Fr\'echet differentiable for sufficiently small $r$. The differentiability is a crucial property to derive first-order optimality conditions which  are usually the starting point for characterizing optimal controls and numerical schemes to solve such type of optimal control problems. While the formal derivation of these optimality conditions for similar settings has been considered, see below, we leave the rigorous derivation of optimality conditions for this specific case for future work.
Difficulties in the analysis to derive Fr\'echet differentiability arise from the fact that (i) we consider the nonlinear Navier-Stokes equation, (ii) the problem is formulated in a polygonal domain, (iii) we have mixed Dirichlet-Neumann boundary conditions, and (iv) the analysis is considered in a higher regularity setting.
Differentiability of FSI problems with respect to data has been considered for the Stokes equation with Dirichlet boundary conditions  in smooth domains coupled with linear elasticity in Wick and Wollner \cite{MR3959888}. 
There the differentiability is obtained by the implicit function theorem which we apply also here following  their ideas. Therefore, the linearized Navier-Stokes operator needs to be an isomorphism in suitable spaces; hence, main parts of the paper deal with the derivation of regularity results for this equation. We proceed in three steps following the procedure in \cite{MR3825153}: (i) Derivation of a lower regularity result for the velocity pressure pair in $W^{1,2}\times L^2$ based on Lax-Milgram arguments, (ii) derivation of a higher regularity result in $W^{2,2}\times W^{1,2}$ which uses estimates from \cite{MR3825153} which relies on results from the Agmon, Douglis, and Nirenberg \cite{MR125307} theory on ellitpic systems, (iii) higher $p$-integrability, namely $W^{2,p}\times W^{1,p}$ on compact subsets using commutator analysis. For the analysis of linear elasticity we rely on classical theory.

 \textbf{We give an overview about related literature.} \emph{On FSI problems}: 
 Galdi and Kyed \cite{MR2563626} analyze existence of steady FSI problems in smooth domains.
Wick and Wollner \cite{MR3959888}
derived as mentioned the differentiability of steady FSI problems with respect to the problem data in smooth domains.
For an introduction to evolutionary FSI problems we refer to Kaltenbacher et al \cite{MR3822723}; moreover, see,  e.g., Gunzburger et al.  \cite{MR1974530,MR2051054}, 
Grandmont and Maday~\cite{MR1763528},   and   Ignatova, Kukavica, Lasiecka, and  Tuffaha \cite{MR3604365}.

\emph{On optimal control and FSI:} In 
\cite{MR3825153} boundary control of a FSI problem with stationary Navier-Stokes equation is considered. The authors show existence of a unqiue solution of the underlying equation under a smallness condition as well as of an optimal control.
 This paper extends  
 Grandmot~\cite{MR1891075} in the sense that the problem is considered in a domain with corners and with mixed boundary conditions are allowed. In the later reference an  elastic body surrounds the fluid and an additional volume constraint is imposed while in the former paper the elastic body is surrounded by the fluid, furthermore, a radial unbounded cost is considered. 
 Rigorously derived first order optimality conditions have been, to the best knowledge of the authors, not been stated yet for the problem under consideration.
 Numerics including formally derived optimality conditions are considered, e.g., in Richter and Wick \cite{MR3111656} where optimal control and parameter estimation for stationary FSI problems are considered.
  
  For control of evolutionary FSI problems see, e.g. Feiler, Meidner, and Vexler \cite{MR3513262} who consider linear FSI systems with coupled linear Stokes equation and wave equation and derive optimality conditions and Moubachir and Zolesio \cite{MR2193461} who derive for an optimal control problem for nonlinear time-dependent FSI problem  necessary optimality conditions formally. Existence of optimal control for the problem of minimizing flow turbulence in the case of a nonlinear fluid-structure interaction models  is considered in Bociu et al. \cite{MR3462441}.

Finally, we remark that differentiablity properties of shape optimization problems for fluid-structure interation has been considered in Haubner, Ulbrich, and Ulbrich \cite{Haubner_2020}.

\textbf{Notation: } Throughout the paper we use the usual notation for Lebesgue and Sobolev spaces. For spaces of  type $W^{s,p}(\Om)^2$ ($W^{s,p}(\Om)^{2\times 2}$ resp.) we often omit the dimension.
We define the symbolic expression 
\be
(w\cdot \nabla) w:=(w_i \partial_i w_1, w_i \partial_i w_2)
\ee
for $w\in W^{1,2}(\Om_2)^2$  
using Leibniz summation convention, and we write 
 $\dddiv w :=\partial_1 w_1 + \partial_2 w_2$. We denote $\nabla \cdot \sigma :=\left(\sum_{j=1}^2 \frac{\partial \sigma_{ij}}{\partial x_j} \right)_{1\le i \le 2}$ for $\sigma \in W^{1,2}(\Om_2)^{2\times 2}$. For matrices $B_1$ and $B_2$ in $\R^{2\times 2}$ we denote the Frobenius product by $A \cdot B:=\sum_{i,j=1}^2 A_{ij} B_{ij}$.
Sometimes we write $0$ for the zero map. The dependence of a function $f$ on another function $g$ is indicated by $f[g]$ 
while the dependence on the spatial variable $x$ by $f(x)=f[g](x)$. 
We use the following notation for the Jacobian of the flow map $\Phi$ as a function of $u$
\be
\nabla \Phi:=\nabla \Phi[u] := D \Phi^{\top}[u] :=\begin{pmatrix}
                             \partial_1 \Phibb_1 &  \partial_1 \Phibb_2 \\
                             \partial_2 \Phibb_1 &  \partial_2 \Phibb_2
                            \end{pmatrix}[u]
\ee
and for the cofactor matrix and determinant of the Jacobian
\be
\begin{aligned}\label{def:K}
K&:=K[u]:=K(D\Phi[u]) := \det (D\Phi[u]) D \Phi[u]^{-\top}=:\cof(D\Phi[u]),\\
J&:=J[u]:=J(D\Phi[u]):=\det(D\Phi[u]).
\end{aligned}
\ee
Moreover, we set
\be
\begin{aligned}
A:=A[u]:=J[u]^{-1}K[u]K[u]^{\top}.
\end{aligned}
\ee
Further, we use the notation
\be\label{A-normal}
\partial_{A[u],n} w  :=  (A[u]\nabla w) \cdot \nx
\ee
with outer normal $\nx$ to $\Om_2$.
With 
\begin{align}
c[u](v,w,z):=((v\cdot K[u]^{\top} \nabla) w,z)_{L^2(\Om)}
\end{align}
we simplify the notation for the case $u$ equal zero to $c(\cdot,\cdot,\cdot):=c[0](\cdot,\cdot,\cdot)$. We set for matrix $K\in \R^{2,2}$
\be
\dddiv_{K^{\top}} w := (K \nabla)^{\top} w.
\ee
For functions $f$ and $e$ and operators $D$ we write for the commutator $[f,D]e:=fDe+D(fe)$. The space of linear bounded mappings from Banach space $X_1$ to Banach space $X_2$ we denote by $L(X_1,X_2)$.

The ball of radius $r>0$ around zero in a Banach space $W$ we denote by $B_r(W)$.
Finally, $c>0$ denotes a generic constant and $c_{\varepsilon}>0$ a constant depending on $\varepsilon>0$. The Euclidean norm in $\R^d$ is denoted by $\norm{\cdot}{}$. 

\smallskip
Structure of the paper: In Section \ref{sec:domain} we introduce the physical setting as well as the flow map and transformation rules between the physical and reference domain, in Section \ref{sec:existence} we introduce the Navier-Stokes system, the elasticity system, and the fluid-structure interaction system and prove existence of solutions, in Section \ref{sec:linear} we show existence and a priori estimates for the linearized system in higher Sobolev norms, and in Section \ref{sec:diff} we show the differentiability of the control to state mapping for the FSI system. In the appendix we recall the transformation of the Navier-Stokes equation and its linearization to the reference domain.

\section{The domain}\label{sec:domain}
We recall the problem setting from Lasiecka et al. \cite{MR3825153}.
Let $D\subset \R^n$, $n=2,3$, be a bounded domain with piecewise regular boundary~$\partial D$ and straight corners as shown in Figure \ref{fig1}. Further, let $\Om_1$ and $\Om_2$ be subsets of $D$ with $\Om_1$ being a doughnut-like domain with boundary $\partial \Om_1:=\Gammaint\cup \Gamma_1$. 
The exterior boundary of $\Om_2$ is denoted by $\Gammaext:=\Gammain\cup  \Gammawall \cup \Gammaout$.
\begin{figure}
 \resizebox{1\textwidth}{!}{
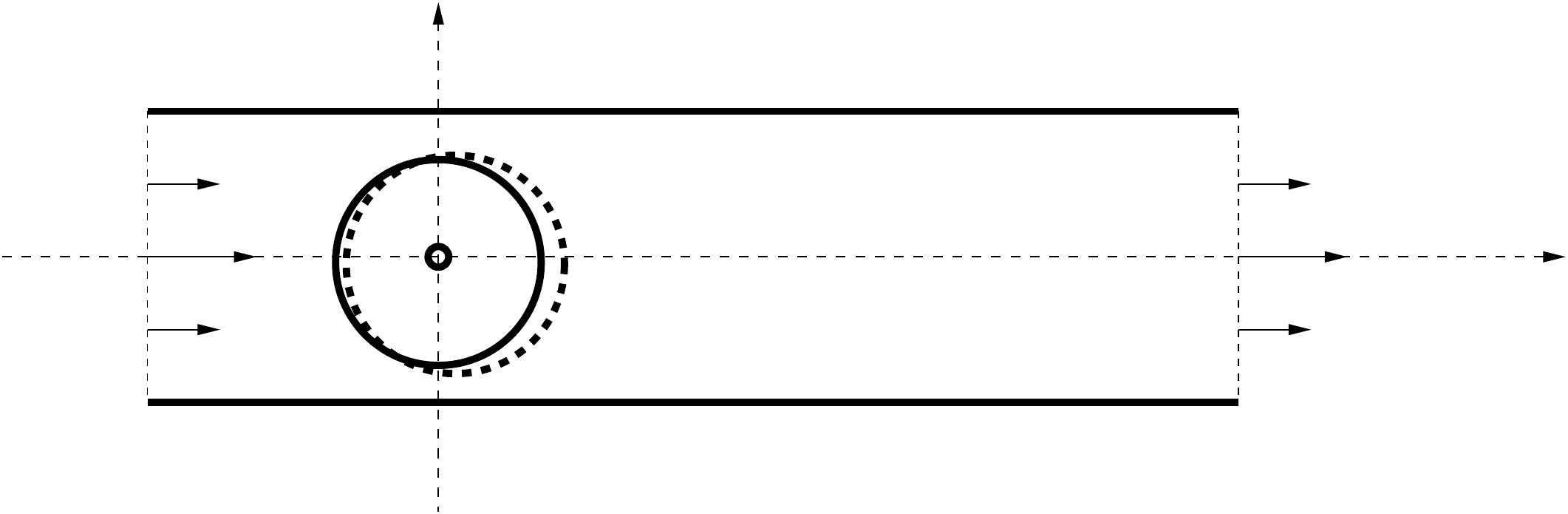
 }
 \caption{Domain.}\label{fig1}
\end{figure}
In $\Om_1$ we consider a \emph{problem of linear elasticity} for an elastic body with $u$ denoting the displacement field. In the exterior subdomain $\Om_2$ we consider a \emph{Navier-Stokes problem} for the motion of a fluid with velocity field denoted by~$\wt$.

 We consider a parallel fluid flow in the channel $D$ containing the elastic body in $\Om_1$ which deforms 
due to the influence of surface forces by the fluid. The original boundary $\Gammaint=\Gammaint[0]$ of $\Om_1$ transforms itself into $\Gammaint[u]$ with elastic displacement $u$ on $\Gammaint$, more precisely
\be
\Gammaint[u]\colon \Gammaint\rar D,\quad x\mapsto x + u(x).
\ee
This leads to a new domain $\Om_2[u]$ with boundaries $\Gammain$, $\Gammaout$, $\Gammawall$, and $\Gammaint[u]$. Variables in the physical domain are denoted with a tilde, cf. Table \ref{table1}. The outer normal to $\Om_2$ is denoted by $\nx$ and the one to $\Om_2[u]$ by $\ny$. The outer normal to $\Om_1$ is denoted by $n_1$.

\begin{table}
\centering
\begin{tabular}{llll}
 & \multicolumn{2}{c}{Domains} & Variables \\
  \cmidrule(lr{.75em}){2-3}  \cmidrule(lr){4-4} 
 \multicolumn{1}{c|}{Original domain} & $\Om_1$ & \multicolumn{1}{c|}{$\Om_2$} & \\
 \hline
 \multicolumn{1}{c|}{Physical domain} & $\Om_1[u]$ & \multicolumn{1}{c|}{$\Om_2[u]$} & $(\tilde{w},\tilde{p})$\\
 \hline
 \multicolumn{1}{c|}{Reference domain} &$\Om_1$ & \multicolumn{1}{c|}{$\Om_2$} & $(w,p)$
\end{tabular}
\smallskip
\caption{Variables in physical and transformed domain.}\label{table1}
\end{table}

\subsection{The flow map and some transformation rules}
In this section we introduce the flow map and study the transformation between the physical and reference domain.
At first, we recall some standard operators.
    The trace operator (cf. \cite[Thm. B.54]{MR2050138}) 
    \be
    \gamma \colon W^{2,p}(\Om_1) \rar W^{2-1/p,p}(\Gammaint),
    \quad \quad 2\le p < \infty,
    \ee
    is surjective and satisfies for $u \in W^{2,p}(\Om_1)$
    \be\label{est:trace}
    \norm{\gamma u}{W^{2-1/p,p}(\Gammaint)} \le c \norm{ u}{W^{2,p}(\Om_1)}.
    \ee
The corresponding trace operator for any open subset $\omega \subset \Gammain\cup \Gammawall$ we denote by~$\gamma_{\omega}$.

\begin{proposition}[Dirichlet harmonic extension]\label{prop:harm_ext}
For $2\le p<\infty$ the harmonic extension
\be\label{phi-i}
\cald\colon W^{2-1/p,p}(\Gammaint) \rar W^{2,p}(\Om_2),\quad \eta_i \mapsto \cald u_i =:\phi_i[\eta_i],\quad i=1,2
\ee
defined by
\be
\label{equ1-c}
\begin{aligned}
\Delta \phi_i &=0 \text{ in }\Om_2,\quad 
 \phi_i =\eta_i  \text{ on }\Gammaint,\quad 
\phi_i =0\text{ on }\partial \Om_2\setminus \Gammaint
\end{aligned}
\ee
is well-posed and satisfies the estimate
\be\label{est:extension}
\norm{\phi_i}{ W^{s,p}(\Om_2[\eta])}\le C \norm{\gamma_{\Gammaint} \eta_i}{W^{s-1/p,p}(\Gammaint)},\quad \text{for } i=1,2.
\ee
\end{proposition}
\begin{proof}
We refer, e.g., to Casas, Mateos, and Raymond \cite[Lem. A.2]{MR2567245}.
\end{proof}

In the following we set $
\phi[\eta]:=(\phi_1[\eta_1],\phi_2[\eta_2])^{\top}$ for $\phi_i$ defined in~\eqref{phi-i}.

Throughout the paper let the integrability exponent $p>2$.
\begin{definition}[Flow map]
For $u\in W^{2,p}(\Om_1)$, and $\phi$ defined in~\eqref{equ1-c} the flow map is given by
\be\label{flow-map}
\Phi\colon W^{2,p}(\Om_1)\rar W^{2,p}(\Om_2[u]),\quad \Phi[u] := \id + \phi(\gamma_{\Gammaint}u).
\ee
\end{definition}
Here, $\Phibb[u](x)$ lifts the boundary trace $u|_{\Gammaint}=\gamma_{\Gammaint}u \in W^{1-1/p,p}(\Gammaint)$ from the interface $\Gammaint$ into $\Om_2[u]=\Phibb(\Om_2)$, in particular we have $\Om_2=\Om_2[0]=\Phibb^{-1}(\Om_2[u])$.

\if{Let $u=u(x)$ be the elastic displacement in $\Om_1$ and let
\be
\Phibb[u]\colon \Om_2 \rar \Om_2[u], \quad x\mapsto x+\Phibb(x) =x + (\Phibb_1(x),\Phibb_2(x),\Phibb_3(x))
\ee
with $\Phibb(x):=\Phibb_i[u_i](x)$ be a harmonic extension (also called Dirichlet map) of the boundary data $u_i\in XXX$
\be\label{equ-1}
\left\{
\begin{aligned}
\Delta \Phibb_i &=0&&\text{in }\Om_2,\\
 \Phibb_i &=u_i&&\text{on }\Gammaint,\\
\Phibb_i &=0&&\text{on }\partial \Om_2\setminus \Gammaint.
\end{aligned}
\right.
\ee
}\fi
We define $U^p:=W^{2,p}(\Om_1) $.
  
   From Grandmont \cite{MR1891075} we recall the following properties stated there for a three dimensional spatial setting.
     
\begin{lemma}\label{grandmont}     
(i) The mapping $ K \colon W^{2,p}(\Om_1) \rar W^{1,p}(\Om_2)$
\be
K[u]:=\operatorname{cof} (\nabla \Phi[u])
\ee
is of class $C^{\infty}$ with cofactor defined in \eqref{def:K}. 

(ii) The mapping $G \colon W^{2,p}(\Om_1) \rar W^{1,p}(\Om_2)$
\be
G[u]:= \nabla \Phi[u]
\ee
is of class $C^{\infty}$. There exists a $r_1 > 0$ such that for all $u \in B_{r_1}(U^p)$
 we have
\be
G[u]=\nabla(\id + \cald(\gamma_{\Gammaint}(u)) = \id + \nabla(\cald(\gamma_{\Gammaint}(u)))
\ee
 is an invertible matrix in $ W^{1,p}(\Om_2)$. Moreover, we have
 \begin{itemize}
 \item[(ii.a)] $\Phi(u) = \id + \cald(\gamma_{\Gammaint} (u))$ is injective on $\overline{\Om_2}$,
\item[(ii.b)] $\Phi(u)\colon \Om_2 \rar \Phi[u](\Om_2)$ is a $C^1$-diffeomorphism.
\end{itemize}
(iii) The mapping $A \colon B_{r_1}(U^p) \rar W^{1,p}(\Om_2)$, with 
\be
 A[u]:=(\nabla (\phi[u]))^{-1} \operatorname{cof} (\nabla(\phi[u]))
\ee
is of class $C^{\infty}$.

 Moreover, $A$ satisfies a condition of \emph{uniform ellipticity} over $B_{r_1}(U^p)$, i.e. there exists a constant $\beta > 0$ such that
\be
A(u)(x) \ge \beta \id,\quad  \text{for all } u \in B_{r_1}(U^p),\quad \text{ and all }x \in \Om_2.
\ee
\end{lemma}
\begin{proof}
(i) The mapping $K[u]$ belongs to $W^{1,p}(\Om_2 )$ since $W^{1,p}(\Om_2)$ is an algebra for $p > 2$ (see Lemma \ref{lem:algebra} with $p = q$).
As a composition of $C^{\infty}$ mappings it is smooth.
(ii) For the first statement we apply the same arguments as in (i). For the second, 
we use that
\be
\Phi(u) = \id + \cald(\gamma_{\Gammaint} (u)) \in W^{2,p}(\Om_2),\quad  \forall u \in W^{2,p}(\Om_1).
\ee
Choosing ${r_1}$ such that
\be
\norm{u}{W^{2,p}(\Om_s)} \le {r_1} \quad \text{implies} \quad \norm{\nabla (\cald(\gamma_{\Gammaint} u ))}{W^{1,p}(\Om_2)} <\frac{1}{c},
\ee
where $c$ is the constant in Lemma \ref{lem:algebra}, then $\id + \nabla (\cald(\Gammaint(b)))$ in an invertible matrix
in $W^{1,p}(\Om_s)$ and we get the result.

For the proof of (ii.a) and (ii.b) we refer to Grandmont \cite[Lem. 2]{MR1891075}.

(iii) We recall the ideas from \cite[Lem. 3]{MR1891075}. Let $b \in B_p$. That
$A[u] \in W^{1,p}(\Om_2)$ follows from point (ii). 
As for the regularity of $A$, it is sufficient
to show that the mapping:
\be
W^{1,p}(\Om_2 ) \rar W^{1,p}(\Om_2 ),\quad  T \mapsto  T^{-1}
\ee
is infinitely differentiable at any invertible matrix of $W^{1,p}(\Om_2)$. This can be proven
by standard arguments, see \cite[Chap. I]{MR0223194}. The condition of uniform ellipticity
of $A$ over $B_{r_1}(U^p)$ derives from continuity and compactness arguments $(W^{1,p}(\Om_2)$ is
compactly embedded in $C(\bar{\Om}_2)$).

For the estimate for the derivative we use the boundedness of $A$ on the bounded set $B_{r_1}(U^p)$.
\end{proof}

\subsection{Transformation of integrals}

We recall some properties on the transformation of integrals and derivatives under a reference map.

For function $\tilde{\pi}$ on the physical domain $\Om_2[u]$ we define the transformed function on the reference domain $\Om_2=\Phi[u]^{-1}(\Om_2[u])$ (for given $u$) by
\be
\pi(x):=\tilde{\pi}(y),\quad y=\Phi[u](x)
\ee
which is well-defined by Lemma \ref{grandmont} (ii).
Moreover, we denote the determinant of the gradient of the flow map by
\be
J(\cdot):=\det(D \Phi(\cdot)).
\ee
As a direct consequence we have $A[u]=J[u]^{-1} K[u]^{\top} K[u]$.

\begin{lemma}\label{lem:rel} Let $u\in B_{r_1}(U^p)$ and $\Phi$ be defined by Proposition \ref{flow-map}. 
Then,  the following relations hold:

(i) Volume elements transform as
\begin{align}\label{lem:rel-1}
 \int_{\Omega_2[u]} 1 \dd y &= \int_{\Om_2} J(x) \dd x,
 \end{align}
 
 (ii) Boundary elements transform with $J_{\Gamma}[u]:=\norm{K[u]\nx}{}$ as
 \begin{align}
 \label{lem:rel-2}
 \int_{ \Gammaout[u]} 1 \dd s_y &= \int_{  \Gammaout}J_{\Gamma}[u]\dd s_x.
 \end{align}
 
 (iii) The gradient transforms as
 \begin{align}\label{lem:rel-3}
 \nabla \ft(y)&=D\Phi^{\top} \nabla f(x)\quad \text{iff} \quad \nabla=\frac{1}{J}K\nabla.
 \end{align}
 
 (iv) For the outer normal $\ny$ to $\Om_2[u]$ and $\nx$ to $\Om_2$ we have
 \begin{align}\label{outer-normal}
 \ny&=\frac{D\Phi^{-\top}\nx}{\norm{D\Phi^{-\top}\nx}{}}=\frac{K\nx}{\norm{K\nx}{}}, 
 \end{align}
 \begin{align}
 \label{item-6}
 \int_{\Gammaint[u]}\tilde{p}(y) \ny \dd s_y & = \int_{\Gammaint} p(x) \frac{\cof(\nabla \phi[u])\nx }{\norm{\cof(\nabla \phi[u])\nx}{}} \norm{\cof(\nabla \phi[u])\nx}{} \nx \dd s_x. 
\end{align}
\end{lemma}
\begin{proof}
 We refer to \cite[Appendix A.1]{MR3825153}.
\end{proof}

\subsection{Transformation of the Navier-Stokes equation}

We consider the Navier-Stokes system in $\R^2$ with viscosity $\nu>0$. 
We define
\be\label{space-g}
\begin{aligned}
\calg_{\mu}&:=\left\{g \in W^{\mu,2}(\Gammain)  \;:\; g|_{\partial \Gammain}=0 \right\},\quad\text{for } \mu\in \left\{\frac{1}{2},\frac{3}{2}\right\}.
\end{aligned}
\ee
Let $\wt=(\wt_1,\wt_2)^{\top}$ the fluid velocity and $\pt$ the pressure in the physical domain $\Om_2[u]=\Phi[u](\Om_1)$ satisfying
\be
\left\{
\begin{aligned}
-\nu \Delta_x\wt_1+ \wt^{\top}\nabla\wt_1 + (\nabla\pt)_1&=0&&\text{in } \Om_2[u],\\
-\nu \Delta_x\wt_2+ \wt^{\top}\nabla\wt_2 + (\nabla\pt)_2&=0 &&\text{in } \Om_2[u],\\ \label{NaS-3}
\nabla^{\top} \wt&=0&&\text{in }\Om_2[u],\\
w&=g &&\text{on } \Gammain,\\
\wt&=0 &&\text{on }\Gammawall \cup \Gammaint[u],\\
-\nu D\wt \cdot \ny+\pt \cdot \ny&=0 &&\text{on }\Gammaout
\end{aligned}
\right.
\ee
and given data  $g\in \calg_{1/2}$.
Let $\Gamma_{\text{bd}}:=\Gammain\cup \Gammawall\cup \Gammaout$, 
and we have by \eqref{equ1-c} $\Phi=\id_x$ on $\Gamma_{\text{bd}}$ 
such that for trial functions $\psit_1$ and $\psit_2$ vanishing on $\Gamma_{\text{bd}}$ also the transformed $\psi_1$ and $\psi_2$ vanish on $\Gamma_{\text{bd}}$. 
The transformed strong form of the Navier-Stokes system in $\Om_2$ is given by (cf. \cite[Appendix~A.1]{MR3825153}), see also Appendix~\ref{app:A},
\be\label{system-rhs-zero}
\left\{
\begin{aligned}
 - \nu \nabla(A[u]\nabla w ) + w (K[u]\nabla)w + K[u]\nabla p &= 0 &&\text{ in } \Om_2,\\
 (K[u]\nabla)^{\top} w &= 0&&\text{ in } \Om_2,\\
w &= g&&\text{on } \Gammain,\\
w &= 0&&\text{on  } \Gammawall \cup \Gammaint,\\
-\nu (A[u]\nabla w) \cdot \nx + pK[u] \cdot \nx &= 0&&\text{on } \Gammaout.
\end{aligned}
\right.
\ee
Since $\Phi=\id_x$ on $\Gammaext$ we have $K \nx=\nx$ on $\Gammaout$.

\section{Existence of solutions for the considered systems}\label{sec:existence}
In this section we consider the nonlinear Navier-Stokes system, the linear elasticity system, as well as the fluid-structure interaction model.

\subsection{The Navier-Stokes system}
%
For  $m=0,1,2$ we introduce 
\be
\begin{aligned}
 \widehat{W}^{m,p}(\Om_2)&:=\{ v\in W^{m,2}(\Om_2)\;:\; v\in W^{m,p}(\widehat{\Om}_2) \text{ for } \widehat{\Om}_2\subset \Om_2\text{ compact}\},
 \end{aligned}
 \ee
and further the spaces, 
 \be
\begin{aligned}
 W^p&:= \widehat{W}^{2,p}(\Om_2) \times \widehat{W}^{1,p}(\Om_2),&
 W&:= W^{2,2}(\Om_2) \times W^{1,2}(\Om_2).
\end{aligned}
\ee
For a given compact subset $\widehat{\Om}_2\subset \Om_2$ we write
\be
\begin{aligned}
W_{w,\widehat{\Om}_2}&:= W^{2,p}(\widehat{\Om}_2)\cap W^{2,2}(\Om_2), 
 & W_{p,\widehat{\Om}_2}&:=W^{1,p}(\widehat{\Om}_2)\cap W^{1,2}(\Om_2);\\
 \Wp &:= \Ww \times  \Wpp;
 \end{aligned}
\ee
note the different meaning of $p$ here as upper and lower index.

\begin{theorem}\label{thm:ex-flow}
  One can choose $r>0$,  $r_1>0$, and $r_2>0$ such that for all $g \in B_r(\mathcal{G}_{3/2})$ and $u \in B_{r_1}(U^p)$ there exists a unique solution $(w,p)$ in $B_{r_2}(W^p)$  of \eqref{system-rhs-zero}. Moreover, for any compact subset $\widehat{\Om}_2\subset \Om_2$  the solution $(w,p)\in B_{r_2}(\Wp)$ depends continuously on~$g$.
 
 \if{(ii) The mapping
 \be
 \mathcal{G} \rar W^p_{\cap},\quad g \mapsto (w[g],p[g])
 \ee
 is continuous.
 
 (iii) The mapping
 \be
 U^p \rar W,\quad u \mapsto (w[u],p[u])
 \ee
 is continuous.}\fi
\end{theorem}
\begin{proof}
We follow closely ideas from \cite{MR3825153}. We consider the fixed point equation
\be
\caln_g\colon B_{r_1}(\Wp) \rar B_{r_1}(\Wp),\quad (w,p)=\caln_g(\wb,\pb),
\ee
where $\caln_g$ maps for given $g\in B_r(\mathcal{G}_{3/2})$ the point $(\wb,\pb)$ to the solution $(w,p) $ of 
\be\label{equ:10010}
\left\{
\begin{aligned}
  -\nu \nabla ( \nabla w ) + \nabla p  & =  -\nu \nabla ( \nabla (-A[u] + \id)\wb )  \\
 & \quad - (\wb (K[u] - \id) \nabla) \wb  - (K[u] - \id) \nabla \pb   && \text{in } \Om_2,\\
 \dddiv w&=-((K[u] - \id) \nabla)^{\top} \wb &&\text{in }\Om_2,\\
 w &=g&&\text{on }\Gammain,\\
 w &=0&&\text{on }\Gamma_{\text{wall}} \cup \Gamma_{\text{int}},\\
 -\nu \partial_{n}  w + p\cdot \nx&=\nu (\partial_{A[u],n} -\partial_{n}) \wb - \pb (K[u] - \id) \cdot \nx  &&\text{on } \Gamma_{\text{out}}.
\end{aligned}
\right.
\ee
Existence follows by Banach's fixed point theorem, see \cite[(68),(85)]{MR3825153}, using smallness of the data $g$.

The continuous dependence on the data follows by the contraction property of $\caln_g$ and the continuous dependence of the iterates on $g$.
\if{
(ii) \emph{Continuous dependency on $g$:}
Let $(g_k)\subset \mathcal{G}$ be a sequence with $g_k \rar g \in \mathcal{G}$. And let $r>0$ be sufficiently small such that $\caln_g$ and $T_{g_k}$ have unique fixed points $(w_g,p_g)$ and $(w_k,p_k)$ in $B_r(W^p)$.
Since
\be
\norm{\caln_g((w,p)) - T_{g_k}((w,p))}{W^p} \le c \norm{g - g_k}{W^{3/2,2(\Gammain)}}\quad \text{uniform in } (w,p) \in W^p
\ee
by \cite[Lem. 4.3, 4.4 and (42)]{MR3825153}, for any $\varepsilon>0$ there exists a $k_0\in \NNN$ such that with contraction constant $0<\eta <1$ and for all $k\ge k_0$ we have
\be
\begin{aligned}
\norm{(w_k,p_k) - (w_g,p_g)}{W} &= \norm{T_{g}(w_g,p_g) -  T_{g_k}(w_k,p_k)}{W}\\
& \le  \norm{T_{g}(w_g,p_g) -T_{g}(w_k,p_k) }{W} + \varepsilon \\
& \le \eta \norm{(w_g,p_g) -(w_k,p_k) }{W}+ \varepsilon.
\end{aligned}
\ee
Hence, we have
\be
\begin{aligned}
\norm{(w_k,p_k) - (w_g,p_g)}{W} &\le \varepsilon / (1-\eta)
\end{aligned}
\ee
showing the continuity of the map.

(iii)  \emph{Lipschitz-continuous dependency on $u$:} 
The proof is similar to (ii). Let $(u_k)\subset B_{r_1}(U^p)$ be a sequence with $u_k \rar u \in B_{r_1}(U^p)$. And let $r>0$ be sufficiently small such that $T(\cdot,\cdot,u)$ and $T(\cdot,\cdot,u_k)$ have unique fixed points $(w_u,p_u)$ and $(w_k,p_k)$ in $B_r(W^p)$.
Since by the continuous dependence of $A[u]\in W^{2,p}(\Om_2)$ and $K[u]\in W^{1,p}(\Om_2)$ on $u$ we have  using \eqref{W2-est} with respect to system \eqref{equ:10010} that\footnotetext{3 Arten von Operatoren $T$.}
\be
\begin{aligned}
&\norm{T_{u}(w_u,p_u) -  T_{u_k}(w_u,p_u)}{W}  \le c\bigg( \norm{\nabla ( \nabla (A[u] - A[u_k])w_u )}{L^2(\Om_2)} \\
&\quad + \norm{(w_u (K[u] - K[u_k]) \nabla) w_u}{L^2(\Om_2)}  + \norm{(K[u] - K[u_k]) \nabla \pb}{ }\\
&\quad + \norm{((K[u] - K[u_k]) \nabla)^{\top}w_u}{W^{1,2}(\Om_2)}  +\norm{-(\partial_{A[u],n} - \partial_{A[u_k],n}) w_u}{W^{1/2,2}(\Gammaout)}\\
&\quad + \norm{\pb (K[u] - K[u_k]) n }{W^{1/2,2}(\Gammaout)}\bigg) \\
& \le c \norm{u - u_k}{U^p}
\end{aligned}
\ee
using trace estimates.
Thus, we have 
\be
\begin{aligned}
\norm{(w_k,p_k) - (w_u,p_u)}{W} &= \norm{T_{u}(w_u,p_u) -  T_{u_k}(w_u,p_u)}{W}\\
& \quad + \norm{T_{u_k}(w_u,p_u) -  T_{u_k}(w_k,p_k)}{W}\\
& \le c\norm{\delta u}{U^p} +  \norm{T_{u_k}(w_u,p_u) -T_{u_k}(w_k,p_k) }{W} \\
& \le c\norm{\delta u}{U^p} + \eta \norm{(w_u,p_u) - (w_k,p_k) }{W}
\end{aligned}
\ee
for $k$ sufficiently large with the constant in \cite[(68)]{MR3825153} does not depend on $k$.
}\fi
\end{proof}

 \begin{hypothesis}\label{hyp1}
 For given $r_2>0$ let $r>0$ and $r_1>0$ be sufficiently small such that for all $g\in B_r(\calg_{3/2})$ and $u\in B_{r_1}(U^p)$ the Navier-Stokes equation \eqref{system-rhs-zero} has a unique solution $(w,p)$ in $ B_{r_2}(W^p)$.
\end{hypothesis}

\subsection{The elasticity system and the traction force}
We set $B:=\{\zeta \in W^{2,p}(\Om_1)\; : \; \zeta|_{\Gamma_1}=0 \}$
 and define the Neumann harmonic extension
\be\label{N}
N\colon W^{1-1/p,p}(\Gammaint) \rar B, \quad
v\mapsto u=:Nv,
\ee
with $u$ be the solution of 
\be\label{syst:101}
\left\{
\begin{aligned}
-\dddiv \sigma[u]&=0&&\text{in }\Om_1,\\
 \sigma[u] &= \mathcal{A} \varepsilon[u] &&\text{on }\Om_1,\\
u &=0&&\text{on }\Gamma_1,\\
\sigma[u]\cdot n_1 & =v \cdot n_1 &&\text{on }\Gammaint
\end{aligned}
\right.
\ee
with outer normal $n_1$ to $\Om_1$ strain tensor $\varepsilon(u):=\half (\nabla u + \nabla u^{\top})$,  Piola Kirchhoff stress tensor components $\sigma = \sigma_{ij}$, $i, j = 1, 2, 3$, and the elasticity tensor $\cala = a_{ijkl}$, $i, j, k, l = 1, 2$, $c_0 > 0$ with
\begin{align}
&a_{ijkl}\xi_{kl} \xi_{ij} \ge c_0\norm{\xi}{} ^2,\quad \forall \xi_{ij},\quad \xi_{ij} = \xi_{ji},&& \text{(positive definiteness)},\\
&a_{ijkl} = a_{klij} = a_{jikl},\quad  a_{ijkl}\in L^{\infty}(\Om_1) &&\text{(symmetry)},
\end{align}
 vector $n_1$ is the unit outward normal along $\Gammaint$ pointing from $\Om_1$ to $\Om_2$.
 We call $u$ the displacement field and will also consider the system with inhomogeneous right hand side
 
 \be
\left\{
\begin{aligned}
-\dddiv \sigma[u]&=f_1&&\text{in }\Om_1,\\
 \sigma[u] &= \mathcal{A} \varepsilon[u] &&\text{on }\Om_1,\\
u &=0&&\text{on }\Gamma_1,\\
\sigma[u]\cdot n_1 & =v \cdot n_1 &&\text{on }\Gammaint.
\end{aligned}
\right.
\ee

\begin{theorem}\label{thm:Ciarlet}
 (i) For $f_1\in L^p(\Om_1)$ and $v\in W^{1-1/p,p}(\Gammaint)$ system~\eqref{syst:101} has a unique solution $u\in W^{2,p}(\Om_1)$, i.e. the Neumann harmonic extension is well-defined and we have
 \be
\norm{Nv }{W^{2,p}(\Om_1)} \le c\left( \norm{v}{W^{1-1/p,p}(\Gammaint)}+\norm{f_1}{L^p(\Om_1)}\right).
 \ee
 
(ii) Moreover, 
 \be
 S\colon L^p(\Om_1) \times W^{1-1/p,p}(\Gammaint) \rar W^{2,p}(\Om),\quad (f_1 , v ) \mapsto u 
\ee
is continuously differentiable.
\end{theorem}
\begin{proof}
(i) We refer to Ciarlet \cite[Thm. 6.3-6 and p. 298]{MR936420}, 
note that $\Gammaint$ has positive distance to $\Gammawall \cup \Gammaout \cup \Gammain$.

(ii) Follows from the linearity of the mapping.
 \end{proof}

\if{
\begin{align}
 W&:=W^{2,p}(\widehat{\Om}_2)\cap W^{2,2}({\Om}_2),\\
 P&:=W^{1,p}(\widehat{\Om}_2)\cap  W^{1,2}({\Om}_2),\\
 X&:= U\times W \times P.
\end{align}
}\fi

Next, we define the traction force on the interface $\Gammaint$.
\begin{definition}[Traction map]
Let $u\in W^{2,p}(\Om_1)$. 
 The traction force is given by
\be\label{equ:tract-force}
\begin{aligned}
t\colon W^{2,p}(\Om_1) \times W^{1,2}(\Om_2) &\rar W^{1/2,2}(\Gammaint),\\ 
(u,p) &\mapsto t[u,p] :=  p[u] K[u] \cdot \nx \text{ on } \Gammaint
\end{aligned}
\ee
with $p[u]$ the pressure in the solution of the Navier-Stokes equation \eqref{system-rhs-zero} and $K[u]$ given by \eqref{def:K}.
\end{definition}
Since $p\in W^{1,p}(\widehat{\Om}_2)$ for compact subsets $\widehat{\Om}_2\subset \Om_2$ and $K[u]$ in $W^{1,p}(\Om_2)$ with $p>2$, we have $pK[u]\in W^{1,p}(\widehat{\Om}_2)$ and so $(pK[u])|_{\Gammaint}\in W^{1/2,2}(\Gammaint)$.
Note that on the interface $\Gammaint$ we have  $w=0$.

\subsection{The fluid-structure interation system}

For $g\in B_r(\calg_{3/2})$ 
and $f_1=0$  we can state the fluid-structure interaction model given as
\be\label{FSI-model}
\left\{
\begin{aligned}
 &\eqref{system-rhs-zero} \text{ together with } 
 u=Nt(u,p) \text{ in } \Om_1\\
 &\text{with $N$ and $t$ defined in  \eqref{N} and \eqref{equ:tract-force}}.
\end{aligned}
\right.
\ee
%
%
%
%
\if{\section{The equations}

\textbf{In $\Om_2[u]$}XXX\footnote{or $\Omega_2$?} we introduce a new coordinate system denoted by $y$ and it is the physical domain for the static Navier-Stookes system
\be
\left\{
\begin{aligned}
-\nu \Delta_y w + (w\cdot \nabla) w + \nabla p&=f&&\text{in }\Om_2,\\
\operatorname{div}_{y} w&=f_2&&\text{in }\Om_2,\\
w&=g&&\text{on }\Gamma,\\
w&=0&&\text{on }\Gamma_{\text{wall}} \cup \Gamma_{\text{int}}(u)\\
-\nu \partial_n  w+p n&=f_3&&\text{on } \Gamma_{\text{out}}.
\end{aligned}
\right.
\ee
Here $w[y]=(w_1,w_2,w_3)^{\top}$ is the velocity field, $p[y]$ the pressure, $\nu=1/Re$ the kinematic viscosity of the fluid, $\ny$ is the outward unit normal vector along $\partial \Om_2[u]$. 
\textbf{In $\Om_1$} we consider
\begin{equation}
\left\{
 \begin{aligned}
  -\operatorname{div} \sigma[u]&=0&&\text{in } \Om_1,\\
  \sigma[u]&=\mathcal{A}\varepsilon[u]&&\text{in }\Om_1\\
  u&=0&&\text{in }\Gamma_1,\\
  \sigma[u] \cdot n_x &= t(\pt)\cdot n&&\text{on } \Gammaint.
 \end{aligned}
 \right.
\end{equation}
}\fi
For $\widehat{\Om}_2\subset \Om_2$ compact we introduce for $2< p < \infty$ the spaces
\begin{align} \label{Xp}
 X^p_1:=U^p \times \Wp,\quad 
 X^p_2:=U^p \times W,\quad 
 X^p:=X^p_1 \cap X^p_2.
\end{align}

\begin{theorem}\label{NSE:existence}
For any $\tilde{r}>0$ there exist an $r>0$ such that for $g \in B_r(\calg_{3/2})$ problem \eqref{FSI-model} has a unique solution $(u,w,p)\in B_{\tilde{r}}(X^p)$ which depends continuously on the data. 
\end{theorem}
\begin{proof}
 We refer to \cite[Thm. 3.2]{MR3825153}. The proof uses a fixed-point argument based on estimates which we already cited in the proof of Theorem \ref{thm:ex-flow}. 
\end{proof}

\section{The linearized equations}\label{sec:linear}
In this section we analyze the linearized Navier-Stokes equation in the domain $\Om_2$ and derive regularity results for its solution using techniques from \cite{MR3825153} which are applied there for the Navier-Stokes equation.

  We introduce the spaces
\begin{align}
\calh&:= \{ w \in W^{1,2}(\Om_2)^2\; :  w = 0 \text{ on } \Gammain \cup \Gammaint \cup \Gammawall \}
\end{align}
and recall the property that for $p>2$ we have $W^{1,p}(\Om_2) \subset L^q(\Om_2)$, $1\le q \le \infty$. 
Moreover, for $u\in U^p$ and $(v,w,y)  \in \Pi_{i=1}^3 W^{1,2}(\Om_2)^2$  we define 
\be\label{c}
c[u](v,w,y):=((v\cdot K[u])\nabla w,y)_{L^2(\Om_2)}. 
\ee

\begin{lemma}\label{lem:c} Let $u\in U^p$ and $(v,w,y)  \in \Pi_{i=1}^3 W^{1,2}(\Om_2)^2 $, then \eqref{c} can be estimated as
 \be
 |c[u](v,w,y)|\le c(u) \norm{v}{W^{1,2}(\Om_2)^{2\times 2}} \norm{w}{L^4(\Om_2)^2}\norm{y}{L^4(\Om_2)^2}.
 \ee
\end{lemma}
\begin{proof}
 By Sobolev's embedding we have  $\partial_j v_i \in L^2(\Om_2)$ and the functions $w_j$ and  $y_i$ belong to $L^4(\Om_2)$ and hence,
 \be
 \int_{\Om_2} \left| v_j \partial_j w_i z_i  \right| \dd x \le  \int_{\Om_2} \left|\partial_j v_i   \right|^2 \dd x \int_{\Om_2} \left|w_j\right|^4 \dd x \int_{\Om_2} \left|z_i \right|^4 \dd x
 \ee
 and we conclude.
\end{proof}
  
 %

In the following we write $c(v,w,y)$ for $c[0](v,w,y)$ with $0$ denoting the zero map.

\subsection{Linearized state equation: Coefficients equal to one}

Let 
\be\label{data}
\begin{aligned}
 f \in L^2(\Om_2), \quad f_2\in L^2(\Om_2),\quad f_3 \in W^{-1/2,2}(\Gammaout),\quad \delta g\in \mathcal{G}_{1/2}.
\end{aligned}
\ee
Let $(\wh,\ph)\in W^p$ solution of the Navier-Stokes equation \eqref{system-rhs-zero} be given. We consider the linearized Navier-Stokes system around this point with inhomogeneous right hand side given by
\be\label{equ1-2sss}
\left\{
\begin{aligned}
 -\nu \Delta z_w  + (\wh  \nabla) z_w + (z_w  \nabla) \wh + \nabla z_p &=f && \text{in } \Om_2,\\
 -(\nabla_{y})^{\top} z_w&=f_2&&\text{in }\Om_2,\\
 z_w &=\delta g&&\text{on }\Gammain,\\
 z_w &=0&&\text{on }\Gamma_{\text{wall}} \cup \Gamma_{\text{int}},\\
 -\nu \partial_{n}  z_w + z_p \cdot \nx&=f_3&&\text{on } \Gamma_{\text{out}}.
\end{aligned}
\right.
\ee
%
%
%
%
%
%
Let $F\colon W^{1,2}(\Om_2) \rar \RR$ be given by
\begin{align}
 F (v) &:= \int_{\Om} ( f + f_2) \cdot v \dd x + \int_{\Gammaout} f_3 \cdot v \dd y,\quad v \in W^{1,2}(\Om_2),
\end{align}
and for $\wh \in \calh$ we define $b_{\wh}\colon \calh \times \calh \rar \RR$, by
\be
b_{\wh}(w,v):=\nu \int_{\Om} \nabla w \cdot \nabla v \dd x + c(\wh,w,v) + c(w,\wh,v).
\ee

To address the linearized terms a smallness condition on the velocity $\wh$ is made, see also de los Reyes and Yousept \cite{MR2524234}.
\begin{lemma}\label{lem:coercivity} 
For $r_2>0$ sufficiently small 
we have 
 \be\label{lemma}
 \begin{aligned}
 b_{\wh}(v, v) \ge c \norm{v}{W^{1,2}(\Om_2)}^2\text{ for all } v \in W^{1,2}(\Om_2);
 \end{aligned}
 \ee
 moreover,
 the bilinear form $b_{\wh}(\cdot,\cdot)$ is continuous.
\end{lemma}
\begin{proof}
By Lemma \ref{lem:c} there exists an $\varepsilon>0$ such that
\be
\begin{aligned}
\nu (\nabla v,\nabla v) &+ c(v,\wh,w)+ c(\wh,v,v)  \\
&\quad \ge \norm{\nabla v}{L^2(\Om_2)}^2  - \varepsilon\norm{v}{L^2(\Om_4)}^2 - \varepsilon\norm{\nabla v}{L^2(\Om_2)}\norm{v}{L^2(\Om_4)} \\
&\quad \ge \frac{\nu}{2} \norm{v}{W^{1,2}(\Om_2)}.
\end{aligned}
\ee
The continuity follows again from Lemma \ref{lem:c} and Sobolev's embedding.
\end{proof}
The weak formulation for \eqref{equ1-2sss} is given as follows: Find $z_w \in \calh$ solution of 
\be\label{equ1-2sss-weak}
\left\{
\begin{aligned}
&b_{\wh}(z_w,v) - \int_{\Om_2} z_p \nabla v \dd x = F(v)\quad \text{for all } v \in \calh,\\
& \dddiv z_w= f_2\text{ in } \Om_2,\quad \quad z_w=\delta g \text{ on } \Gammain.
\end{aligned}
\right.
\ee

\begin{theorem}\label{thm:ex-lin}
For 
$\norm{\wh}{W^{1,2}(\Om_2)}$ sufficiently small system \eqref{equ1-2sss} (resp. \eqref{equ1-2sss-weak}) has a unique solution $(z_w,z_p) \in W^{1,2}(\Om_2) \times L^2(\Om_2)$
with
\be
\begin{aligned}
\norm{z_w}{W^{1,2}(\Om_2)}+\norm{z_p}{L^2(\Om_2)} & \le c\norm{f}{W^{-1,2}(\Om_2)}+c\norm{f_2}{L^2(\Om_2)} \\
& \quad + c\norm{f_3}{W^{-1/2,2}(\Gammaout)} + c\norm{\delta g}{W^{1/2,2}(\Gammain)}. 
\end{aligned}
\ee
\end{theorem}
Note, that this lower regularity existence and the estimate follows by classical Lax-Milgram arguments, see \cite[Step 1]{MR3825153} and also \cite[Theorem~11.1.2]{zbMATH05709231},  together with Lemma \ref{lem:coercivity}. 

\begin{hypothesis}\label{hyp2}
 Let $r_2>0$ be sufficiently small such that for $\wh \in B_{r_2}(\widehat{W}^{2,p}(\Om_2))$  equation \eqref{equ1-2sss} has a unique solution $(z_w,z_p) \in W^{1,2}(\Om_2) \times L^2(\Om_2)$. 
\end{hypothesis}
Note, that here we consider a higher norm than necessary with respect to Theorem \ref{thm:ex-lin}. This is due to the fact that later we will also estimate higher norms of $\wh$.

\subsection{The linearized state equation}
Let $(\wh,\ph)$ be given solution of the Navier-Stokes equation \eqref{system-rhs-zero}. 
We consider the in this point linearized equation with inhomogeneous right hand sides chosen as in \eqref{data} given by
\be\label{system-rhs-zero-1}
\left\{
\begin{aligned}
   -\nu \Delta (A[u] z_w ) + \wh (K[u] \nabla) z_w  
   + z_w  (K[u] \nabla) \wh + K[u]  \nabla z_p  
 & =  
 f && \text{in } \Om_2,\\
 (K[u]  \nabla)^{\top} z_w &=  f_2&&\text{in }\Om_2,\\
 z_w &=\delta g &&\text{on }\Gammain,\\
 z_w &=0&&\text{on }\Gamma_{\text{wall}} \cup \Gamma_{\text{int}},\\
- \nu \partial_{A[u],n}  z_w + z_p K[u]\cdot \nx  &= f_3&&\text{on } \Gamma_{\text{out}}.
\end{aligned}
\right.
\ee    
We follow the approach from \cite{MR3825153} where the nonlinear Navier-Stokes equation is analyzed and ideas from Grandmont \cite{MR1891075}.
We recall a technical result which follows by a Taylor argument. 
\begin{lemma}\label{lem:4.1} For $r_u$ and $r_{w}$ positive and $\ub\in B_{r_u}(U^p)$ and $\wb \in B_{r_{w}}(X^p)$ and some $s\ge 1$ the following estimates hold:
 \begin{align*}
 &\text{(i)}\quad  \norm{A[\ub] - \id}{L^{\infty}(\Om_2)} \le cr_u^s,&& \text{(ii)} \quad  \norm{A(\ub) - \id}{W^{1,p}(\Om_2)} \le  cr_u^s,\\
 &\text{(iii)} \quad  \norm{K(\ub)}{L^{\infty}(\Om_2)} \le  c(1 + r_{w}^s), && \text{(iv)}\quad  \norm{K(\ub) - \id}{L^{\infty}(\Om_2)} \le  cr_u^s,
 \end{align*}
 as well as
 \begin{align*}
 \raggedleft
 &(v)\quad \norm{\nabla((A(\ub) - \id)\nabla)\wb}{L^q(\Om_2)} \le  cr_u^s \norm{\wb}{W^{2,q}(\Om_2)}
 \end{align*}
for some $s \ge 1$ and $q \ge 2$.
\end{lemma}
\begin{proof}
 See  \cite[Lem. 4.1]{MR3825153}.
\end{proof}

We follow ideas in \cite[Prop. 4.2, Lem. 4.3, Lem~4.4, and Lem. 4.5]{MR3825153} developed there for the Navier-Stokes equation to analyze the linearized equation in \eqref{system-rhs-zero-1}.
We start with a preliminary consideration which is later used in \eqref{est-bdry}.
\begin{lemma}
 For $v\in W^{1,p}(\Om_2)$ and $s \in L^2(\Gammaout)$ we have
 \be\label{est:infty}
 \norm{vs}{W^{-1/2,2}(\Gammaout)} \le \norm{v}{L^{\infty}(\Om_2)} \norm{s}{W^{-1/2,2}(\Gammaout) }.
 \ee
\end{lemma}
\begin{proof} Since $W^{1,p}\subset C(\overline{\Om})$ continuous the product of the trace of $v$ on $\Gammaout$ with $s$ is in $L^2(\Gammaout)\subset W^{-1/2,2}(\Gammaout)$ and we have
\be
 \begin{aligned}
 \norm{vs}{W^{-1/2,2}(\Gammaout)}  
 &= \sup_{\norm{\eta}{W^{1/2,2}(\Gammaout)}=1} ( |v| |s|, |\eta| )_{L^2(\Gammaout)} \\
 &\le \norm{v}{L^{\infty}(\Gammaout)} \norm{s}{W^{-1/2,2}(\Gammaout) }.
 \end{aligned}
 \ee
 Again using that $v$ is continuous up to the boundary we conclude.
\end{proof}

For given $u\in B_{r_1}(U^p)$ we define a map 
\be\label{map-T}
T=T_u\colon W^p \rar W^p,\quad (\zb_{w},\zb_{p})\mapsto (z_w,z_p)
\ee
by rewriting \eqref{system-rhs-zero-1} as
\be\label{equ1-2q}
\left\{
\begin{aligned}
  -\nu \nabla ( \nabla z_w ) &+ (\wh  \nabla) z_w  + (z_w  \nabla) \wh + \nabla z_p \\
 & =  
 -\nu \nabla ( \nabla (-A[u] + \id)\zb_w )\\
 &\quad - (\wh (K[u] - \id) \nabla) \zb_w  \\
 & \quad - (\zb_w  (K[u] - \id)\nabla) \wh\\
 &\quad -  (K[u] - \id) \nabla \zb_p  + f && \text{in } \Om_2,\\
 \dddiv z_w&=-((K[u] - \id) \nabla)^{\top} \zb_w + f_2&&\text{in }\Om_2,\\
 z_w &=\delta g &&\text{on }\Gammain,\\
 z_w &=0&&\text{on }\Gamma_{\text{wall}} \cup \Gamma_{\text{int}},\\
 -\nu \partial_{n}  z_w + z_p \cdot n_x &= \nu \partial_{A[u]-\id,n}  \zb_w - \zb_p (K[u] - \id)\cdot \nx + f_3&&\text{on } \Gamma_{\text{out}};
\end{aligned}
\right.
\ee
this will allow to define a sequence $((z_{w,n},z_{p,n}))_{n\in \NNN}$ with $(z_{w,0},z_{p,0})$ equal to some $(\zb_{w},\zb_{p})\in W^p$ which we further analyze in Section \ref{sec:limit} to obtain existence of a solution for~\eqref{system-rhs-zero-1}.

\begin{lemma}\label{lem:convection} Let $r_u$ and $r_w$ positive. For $u\in B_{r_u}(U^p)$ and $v \in B_{r_w}(W^{1,p}(\Om_2))$ we have 
\be
\begin{aligned}
 \norm{v ( K(u)\nabla) z_w}{L^p(\Om_2)} 
 &\le c (1+r_u^s)r_{w} \norm{z_w}{W^{1,p}(\Om_2)}.
\end{aligned}
\ee
\end{lemma}
\begin{proof} We have 
 \be
\begin{aligned}
 \norm{v(K(u)\nabla z_w)}{L^p(\Om_2)} 
& \le \norm{K(u)}{L^{\infty}(\Om_2)}\norm{v}{L^{\infty}(\Om_2)} \norm{\nabla z_w}{L^p(\Om_2)} \\
& \le (1+r_u^s)r_{w} \norm{z_w}{W^{1,p}(\Om_2)}
\end{aligned}
\ee
and conclude with Lemma \ref{lem:4.1}. 
\end{proof}
\subsection{Lower regularity}

We have the following a priori $W^{1,2}\times L^{2}$-estimate without having to take into account the special situation of mixed boundary conditions.

\begin{lemma}\label{42}
 Let Hypothesis \ref{hyp2} be satisfied. For the solution $(z_w,z_p)$ of \eqref{equ1-2q} we have the estimate 
 \begin{equation}\label{est-lower}
 \begin{aligned}
 \norm{z_w}{W^{1,2}(\Om_2)} + \norm{z_p}{L^2(\Om_2)} & \le c\norm{f}{L^2(\Om_2)} + c\norm{f_2}{L^{2}(\Om_2)} + c \norm{g}{W^{1/2,2}(\Gammain)} \\
 &\quad + c\norm{f_3}{W^{-1/2,2}(\Om_2}
 +\blue{c r_1^{s}  \norm{\zb_w}{W^{2,2}(\Om_2)}}
 \\
 & \quad + c r_1^s \norm{\zb_p}{W^{1,2}(\Om_2)}
 \end{aligned}
 \end{equation}
 with constant $c$ depending on $r_{2}$ and  $s\ge 1$ and $p>2$.
\end{lemma}
\begin{proof} 
By Theorem \ref{thm:ex-lin} we have existence of a unique solution and the following lower regularity result for the solution $(z_w,z_p)$ given by
 \be
 \begin{aligned}
 &\norm{z_w}{W^{1,2}(\Om_2)} + \norm{z_p}{L^2(\Om_2)} \le c \bigg(  \norm{f}{L^2(\Om_2)} + \norm{f_2}{W^{1,2}(\Om_2)} + \norm{g}{W^{1/2,2}(\Gammain)} \\
 &\quad + \norm{f_3}{W^{-1/2,2}(\Om_2)} + \norm{F(u,\zb_w,\zb_p)}{L^2(\Om_2)} + \norm{F_2(\zb_w,u)}{L^2(\Om_2)}\\
 &\quad + \purple{\norm{\partial_{A[u]-\id,n} \zb_w}{W^{-1/2,2}(\Gammaout)}} 
 + \orange{\norm{\zb_p(K[u] - \id)}{W^{-1/2,2}(\Gammaout)}} \bigg),
 \end{aligned}
 \ee
 where 
 \be
 \begin{aligned}
 F(u,\zb_w,\zb_p) & := \blue{ -\nu \nabla ((-A[u]+\id)\nabla z_w)} \\
 & \quad - \cyan{\zb_w(K[u]-\id)\nabla \wh - \wh(K[u]-\id)\nabla \zb_w}  - \brown{(K[u] - \id)\nabla \zb_p},\\
 F_2(u,\zb_w) & := \dddiv_{(\id - K[u]^{\top})} \zb_w = \rred{((\id - K[u]) \nabla)^{\top}\cdot \zb_w}.
 \end{aligned}
 \ee

 We estimate each term separately. Differently to \cite{MR3825153} we have to estimate the linearized convection term
 \cyan{
  \be\label{eq:1q}
  \begin{aligned}
  \norm{\zb_w(-K[u]-\id)\nabla \wh}{L^2(\Om_2)}
  & \le c \norm{-K[u]-\id}{L^{\infty}(\Om_2)} \norm{\zb_w}{L^{\infty}(\Om_2)} \norm{\nabla \wh}{L^2(\Om_2)} \\
  & \le c r_1^{s} r_{2}  \norm{\zb_w}{L^{\infty}(\Om_2)}
  \end{aligned}
  \ee
  }
  and accordingly,
  \cyan{
  \be\label{eq:1qr}
  \begin{aligned}
  \norm{\wh(-K[u]-\id)\nabla \zb_w}{L^2(\Om_2)}  
  & \le c r_1^{s}\norm{\zb_w}{W^{1,2}(\Om_2)} \norm{\wh}{L^{\infty}(\Om_2)}.
  \end{aligned}
  \ee}
 The other terms are treated in the same way, for simplicity we recall here the main steps.
 For some $s\ge 1$ using Lemma \ref{lem:4.1} 4. we have for the diffusion term
 \begin{align}\label{est:diff}
 \blue{ \norm{\nu \nabla ((A[u]-\id)\nabla z_w)}{L^2(\Om_2)} \le c r_1^s \norm{z_w}{W^{2,2}(\Om_2)}}.
  \end{align}
  Again by \cite[Lem. 4.1]{MR3825153} we obtain for the term involving the pressure
  \begin{align}
  \brown{\norm{(K[u] - \id)\nabla \zb_p}{L^2(\Om_2)}  \le  \norm{K[u] - \id}{L^{\infty}(\Om_2)} \norm{\nabla \zb_p}{L^2(\Om_2)} \le  c r_1^s \norm{ \zb_p}{W^{1,2}(\Om_2)}}.
  \end{align}
  By 
  $\dddiv_{\id- K[u]^{\top}} w = (\id - K[u]^{\top}) \cdot \nabla w$, cf. Appendix \ref{app:prop}, 
  we have 
  \rred{  \begin{align}
  \norm{\dddiv_{\id - K[u]^{\top}}\zb_{w}}{L^2(\Om_2)} &\le \norm{\id - K[u]^{\top}}{L^{\infty}(\Om_2)} \norm{\zb_w}{W^{1,2}(\Om_2)} \le c r_1^s \norm{\zb_w}{W^{1,2}(\Om_2)} 
  \end{align}}
  and for the boundary terms 
 \be
  \begin{aligned}\label{eq:lem-last}
 &\norm{\partial_{A - \id,n}  \zb_w}{W^{1/2,2}(\Gammaout)} 
  \le \norm{A[u] -  \id}{L^{\infty}(\Om_2)}
\norm{\partial_n\zb_{w} }{W^{1/2,2}(\Gammaout)}\\
&\quad + \norm{A[u] - \id}{W^{1,p}(\Om_2)}
\norm{\zb_{w}}{W^{2,2}(\Om_2)}
\le \green{c r_1^s \norm{\zb_{w}}{W^{2,2}(\Om_2)}}
  \end{aligned}
  \ee
  using for the latter estimate the Neumann trace estimate; note, that we estimate the trace in a higher norm than necessary here. Moreover, with estimate \eqref{est:infty} 
  \orange{
  \be\label{est-bdry}
  \begin{aligned}
  \norm{\zb_p (K[u] - \id)}{W^{-1/2,2}(\Gammaout)}  & \le \norm{K[u] - \id}{L^{\infty}(\Om_2)} \norm{\zb_p}{W^{-1/2,2}(\Gammaout)}\\
  & \le c r_1^s \norm{\zb_p}{W^{-1/2,2}(\Gammaout)}.
 \end{aligned}
 \ee
 }
 Consequently, with \eqref{est:diff}--\eqref{eq:lem-last} we obtain the result. 
 \if{
 \be
 \begin{aligned}
 \norm{z_w}{W^{1,2}(\Om_2)} + \norm{z_p}{L^2(\Om_2)} \le & 
 c \norm{f}{L^2(\Om_2)} + c\norm{f_2}{W^{1,2}(\Om_2)} + c\norm{g}{W^{1/2,2}(\Gammain)} \\
 &\quad + c\norm{f_3}{W^{-1/2,2}(\Om_2)}
 \blue{+c r_1^s \norm{z_w}{W^{2,2}(\Om_2)}} 
  \\
 & \quad + \cyan{ c_{r_{23}} r_1^s  \norm{\zb_w}{2,2}  }+\brown{c r_1^s \norm{ \zb_p}{W^{1,2}(\Om_2)}}  \\
 & \quad  + \rred{ c r_1^s \norm{\zb_w}{W^{1,2}(\Om_2)}} + \orange{c r_1^s \norm{\zb_p}{W^{-1/2,2}(\Gammaout)}} 
 \end{aligned}
 \ee
 which shows the result.}\fi
\end{proof}

\subsection{Higher regularity}

For $F\in L^2(\Om_2)$, $F_2\in W^{1,2}(\Om_2)$, $\delta g\in\calg_{3,2}$, and $F_3 \in W^{1/2,2}(\Gammaout)$ we consider 
\be
\left\{
\begin{aligned}
 -\Delta v  + \nabla q &= F,&&\text{in } \Om_2,\\
 \dddiv v &= F_2,&&\text{in } \Om_2,\\
 v &= g,&&\text{in } \Gammain\cup \Gammawall,\\
 -\partial_n v + q\cdot \nx  &= F_3,&&\text{in } \Gammaout.
\end{aligned}
\right.
\ee
Let $\kappa \in  C^{\infty}(\Om)$ localize $v$ away from the external boundary $\partial \Om_2$ and set  $\Omega_{\kappa}:=\operatorname{supp}(\kappa)$. Here, we rely on estimates  provided in  Lasiecka et al. \cite[equation (44)]{MR3825153} given by
\be\label{1-kappa-estimate}
\begin{aligned}
&\norm{(1 - \kappa)v}{W^{2,2}(\Om_2)}  + \norm{(1 - \kappa)q}{W^{1,2}(\Om_2)} \le  c\norm{(1 - \kappa)F}{L^2(\Om_2)}  \\
& \quad  + \norm{(1 - \kappa)F_2}{W^{1,2}(\Om_2)}+ \brown{\norm{g}{W^{3/2,2}(\Gammain)}} + \norm{F_3}{W^{1/2,2}(\Gammaout)}.
\end{aligned}
\ee
\begin{remark}
 The authors in \cite{MR3825153} refer here to the notion of ellipticity for systems introduced in Agmon, Douglis, and Nirenberg \cite{MR125307}, see also Maz'ya and Rossmann \cite[Sec. 1.1.3]{zbMATH05709231}, and  Bouchev and Gunzburger \cite[Appendix D]{MR2490235}. Following Bene\v{s} and Ku\v{c}era \cite[Appendix]{MR3458302} the regularity is established at first locally for boundary points on the Dirichlet boundary part, the Neumann boundary part, and then for the two corners where the different types of boundary conditions meet (the less standard result), see \cite[Appendix A.3]{MR3825153}. With cut-off functions the solutions are localized and the estimates are derived using \cite[Thm. D.1]{MR2490235}. Using the compactness of $\bar{\Om}$  global regularity is achieved.  
\end{remark}
We define
 \be
 \cals^{p'}:= \widehat{W}^{0,p'}(\Om_2)\cap L^2(\Om_2) \times \widehat{W}^{1,p'}(\Om_2) \times \calg_{3/2} \times W^{1/2,2}(\Gammaout), \quad p'\ge 2, 
 \ee
and assume
\be\label{data-high}
\begin{aligned}
 (f,f_2,g,f_3)\in \cals^{p'}.
\end{aligned}
\ee
We introduce
\be
\begin{aligned}
z_{w,a}&:=\kappa z_w, & z_{w,b}&:= (1-\kappa) z_w, \\
z_{p,a}&:=\kappa z_p, & z_{p,b}&:= (1-\kappa) z_p
\end{aligned}
\ee
implying $z_w=z_{w,a}+z_{w,b}$ and $z_p=z_{p,a}+z_{p,b}$ 
and write the solution $(z_w,z_p)$ of \eqref{equ1-2q} as the sum of  $(z_{w,a},z_{p,a})$ and $(z_{w,b},z_{p,b})$ being solutions of the following two systems localized in the interior and close to the boundary:  
\be\label{eq:25}
\left\{
\begin{aligned}
- \nu \nabla(\nabla z_{w,a}) + \nabla z_{p,a}&= - \nu \nabla((-A[\ub] + \id)\nabla \zb_{w,a}) + \nu [\nabla((-A[u] + \id) \nabla ),\kappa]\zb_w \\
&\quad \rred{-\kappa \left(\zb_{w} ((K[\ub])\nabla)\wh\right)}  - \rred{\kappa \wh ((K[\ub])\nabla)\zb_{w}}\\
&\quad - \kappa (K[\ub] - \id)\nabla \zb_p  - [\kappa, \nu \nabla^2]z_w + [\kappa, \nabla]z_p + \kappa f, \\
\dddiv z_{w,a} &= \dddiv_{\id - K[\ub]}\zb_{w,a} + [\kappa, \dddiv]z_w, \\
&\quad + [\dddiv_{\id - K^{\top}[\ub]}, \kappa]z_w+ \kappa f_2,\\
z_{w,a} &=  0\text{ on } \partial \Om_2,
\end{aligned}
\right.
\ee
and 
\be\label{eq:w_b}
\left\{
\begin{aligned}
 - \nu \nabla(\nabla z_{w,b}) + \nabla z_{p,b} &= -(1 - \kappa)\bigg(\nu \nabla((-A[\ub] + \id)\nabla \zb_{w} )\\
&\quad \rred{-\zb_{w} (K[\ub]\nabla)\wh - \wh (K[\ub]\nabla)\zb_{w}}\\
&\quad - (K[\ub] - \id)\nabla \zb_p\bigg)
+ [1 - \kappa, \nu \nabla^2]z_w \\
&\quad - [1 - \kappa, \nabla]z_p +f,\\
\dddiv z_{w,b} &= (1 - \kappa)(\dddiv_{\id -K^{\top}[\ub]}\zb_{w}) + [1 - \kappa, \dddiv]z_w+f_2,\\
z_{w,b} &= (1 - \kappa)\delta g= g\text{ on } \Gammain,\\
z_{w,b} &= 0\text{ on } \Gammawall \cup \Gammaint,\\
-\nu  \partial_n z_{w,b} + z_{p,b}\cdot \nx &= - \partial_{A[\ub] - \id,n} \zb_w + \zb_p (K[\ub] - \id) \cdot \nx +f_3 \text{ on } \Gammaout.
\end{aligned}
\right.
\ee

\begin{lemma}\label{lem:w_a} 
Let Hypothesis \ref{hyp2} be satisfied. For every $\varepsilon > 0$ we have for $p'=2$, and  
$s \ge 1$ that 
\be
\begin{aligned}
\norm{z_{w,b}}{W^{2,2}(\Om_2)}& + \norm{z_{p,b}}{W^{1,2}(\Om_2)} 
 \le c \brown{\norm{(f,f_2,g,f_3)}{\cals^2}}  + \blue{ cr_1^s \norm{\zb_{w,b}}{W^{2,2}(\Om_2)}} \\
 &\quad + \green{ c r_1 \norm{\zb_{p}}{ W^{1,2}(\Om_2)} } 
+ \varepsilon \norm{z_w}{W^{2,2}(\Om_2)}  + c_{\varepsilon} \norm{z_w}{L^2(\Om_2)} \\
&\quad + c\norm{z_p}{L^2(\Om_2)}  + \rred{c r_{2}\norm{\zb_{w}}{W^{2,2}(\Om_2)}}.
\end{aligned}
\ee
\end{lemma}

\begin{proof}
In the following we omit the first term in the estimate on the right hand side, since its derivation follows easily.
By \eqref{1-kappa-estimate} we have for the solution of equation \eqref{eq:w_b}
\be\label{W2-est}
\begin{aligned}
&\norm{(1 - \kappa)z_w}{W^{2,2}(\Om_2)}  + \norm{(1 - \kappa)z_p}{W^{1,2}(\Om_2)} \le  c\norm{F^b}{L^2(\Om_2)}\\
& \quad  + \norm{F^b_2}{W^{1,2}(\Om_2)} + \norm{ F_3^1 }{W^{1/2,2}(\Gammaout)} + \norm{ F_3^2 }{W^{1/2,2}(\Gammaout)},
\end{aligned}
\ee
where
\be
\begin{aligned}
F^b & := (1 - \kappa)\bigg(\nu \nabla((-A[\ub] + \id)\nabla \zb_{w} ) \rred{+ \wh ((-K[\ub])\nabla)\zb_{w}} \\
&\quad \rred{+\zb_{w} ((-K[\ub])\nabla)\wh} - (K[\ub] - \id)\nabla \zb_p\bigg)\\
& \quad + [1 - \kappa, \nu \nabla^2 ]z_w - [1 - \kappa, \nabla]z_p  =: \sum_{i=1}^6 I_i,\\
F^b_2 &:= (1 - \kappa)(\dddiv_{\id - K^{\top}[\ub]}\zb_{w}) + [1 - \kappa, \dddiv]w=: I_7 + I_8,\\
F_3^1 &:= \zb_{p}(K[\ub] - \id) \cdot n_x,\\
F_3^2 &:=  \partial_{ A[\ub] - \id,n}\wb.
\end{aligned}
\ee
Note, 
 that in the following we consider general $L^q$, $q\ge 1$,  and not only $L^2$ estimates to include also estimates needed for the subsequential lemma in which instead of $(z_w,z_p) \in W$ the pair $(z_{w,a},z_{w,b})\in W^p$ will be considered implying that below higher regularity has to be assumed for terms involving $\norm{z_w}{W^{2,p}(\Om_2)}$. 
We have with $q\ge 2$ for the linearized convection term 
using Sobolev embedding $W^{2,2}(\Om_2)\subset W^{1,q}(\Om_2)$
\be
\begin{aligned}
 \rred{\norm{I_2 + I_3}{L^q(\Om_2)}} & \le  \norm{K[u]}{W^{1,q}(\Om_2)}  \norm{\wh}{L^{\infty}(\Om_2)}  \norm{\zb_w}{W^{2,2}(\Om_2)}\\
 &\quad + \norm{K[u]}{W^{1,q}(\Om_2)}\norm{\zb_w}{L^{\infty}(\Om_2) }\norm{\wh}{W^{2,2}(\Om_2)}\\
 & \le c r_{2} (\norm{\zb_{w}}{W^{2,2}(\Omega_2)} + \norm{\zb_{w}}{L^{\infty}(\Omega_2)}).
 \end{aligned}
 \ee
Moreover, following \cite{MR3825153}, with H\"older's inequality with suitable $q_1\ge 1$ and $q_2\ge 1$ satisfying $1/q=1/q_1 + 1/q_2$ and  Lemma \ref{lem:convection}
\be
\begin{aligned}
\norm{I_1}{L^q(\Om_2)} & \le c\norm{A[\ub] - \id}{L^{\infty}(\Om_2)}\norm{\zb_{w}}{W^{2,q}(\Om_2)} \\
&\quad
 + c\norm{A[\ub] - \id}{W^{1,q_1}(\Om_2)}
\norm{\zb_{w}}{W^{1,q_2}(\Om_2)} \\
& \le \blue{cr_1^s
\norm{\zb_{w}}{W^{2,q}(\Om_2)}};
\end{aligned}
\ee
for the later estimate we used that for $q_2=(qq_1)/(q_1 - q)$ the inclusion $W^{2,q}(\Om_2) \subset W^{1,q_1}(\Om_2)$ is continuous. 
Using that the appearing commutator lose one order of differentiability we get
\be
\norm{I_5+I_6}{L^q(\Om_2)} \le c (\norm{z_w}{W^{1,q}(\Om_2)} + \norm{z_p}{L^q(\Om_2) })
\ee
which can be further estimated in the case $q=2$ by \eqref{est-lower}.
Further, we have for $q= 2$ that
\be\label{abcd1}
\begin{aligned}
\norm{I_4}{L^2(\Om_2)} &\le \norm{K[\ub] - \id}{L^{\infty}(\Om_2)} \norm{\nabla z_p}{L^2(\Om_2)} \le c r_1^s \norm{\nabla z_p}{L^2(\Om_2)}.
\end{aligned}
\end{equation}

Note, that the boundary terms are not relevant for the system in the variables $(z_{w,a},z_{p,a})$ considered in the subsequential lemma, so we consider here only the case $q=2$. 
We have with H\"older's inequality 
\be
\begin{aligned}
\norm{F_3^1}{W^{1/2,2}(\Gammaout)} 
& \le c\norm{z_p(K[u] - \id)\cdot \nx}{W^{1/2,2}(\Gammaout)} \\
& \le c\norm{z_p}{W^{1,2}(\Om_2)} \norm{K[u] - \id}{L^{\infty}(\Om_2)}\\
&\quad + \norm{z_p}{L^{\frac{2q}{q-2}}(\Om_2)}
\norm{K[u] - \id}{W^{1,q}(\Om_2)}.
\end{aligned}
\ee
Now, using that the composition for $q\ge 2$
\be
\nabla \circ \cald \circ \gamma_{\Gammaint}\colon W^{2,q}(\Om_1) \rar W^{1,q}(\Om_2), \quad  u \mapsto \nabla \phi(u)
\ee
defines a continuous inclusion we have together with Lemma \ref{lem:4.1} that
\be\label{est:1-p}
\norm{K[u] - \id}{W^{1,q}(\Om_2)} \le c \norm{u}{W^{2,q}(\Om_2)}^s
\ee
and we can conclude 
\be
\begin{aligned}
\norm{F_3^1}{W^{1/2,2}(\Gammaout)} & \le \norm{z_p}{W^{1,2}(\Om_2)}
\left(\norm{K[u] - \id}{L^{\infty}(\Om_2)} + \norm{K[u] - \id}{W^{1,q}(\Om_2)}\right)\\
& \le \green{c r_1^s \norm{z_p}{W^{1,2}(\Om_2)} }.
\end{aligned}
\ee
Next, we have as in \eqref{eq:lem-last} the estimate 
\be
\begin{aligned}\label{eq:abc1}
\norm{F_3^2}{W^{1/2,2}(\Gammaout)}
&\le \green{c r_1^s \norm{\zb_{w,b}}{W^{2,2}(\Om_2)}}.
\end{aligned}
\ee
For  $q \ge 2$ we have using Appendix \ref{app:prop} and \eqref{est:1-p} that 
\be\label{eq:abc2}
\begin{aligned}
 \norm{I_7}{W^{1,q}(\Om_2)}  &\le c\norm{\ub}{W^{2,q}(\Om_1)}
\norm{\zb_{w,b}}{W^{2,q}(\Om_2)} \\
&\quad + \norm{\id - K^{\top}(\ub)}{L^{\infty}(\Om)}\norm{\zb_{w,b}}{W^{2,q}(\Om_2)}\\
&\le \green{c r_1\norm{\zb_{w,b}}{W^{2,q}(\Om_2)}}.
\end{aligned}
\ee
Moreover, we have
\begin{align}
 \norm{I_8}{W^{1,2}(\Om_2)} \le \norm{[1-\kappa,\dddiv]w}{W^{1,2}(\Om_2)}   \le c \norm{w}{W^{1,2}(\Om_2)}
\end{align}
using that $[(1-\kappa),\dddiv ]w = \nabla (1-\kappa) \cdot w$. The norm on the right hand side can be further estimated using again \eqref{est-lower}.

Now, setting $q=2$ we conclude.
\end{proof}

\subsubsection{Interior estimates}
For references on $L^p$--estimates for the Stokes equation we refer to Amrouche and Rejaiba \cite{MR3145765}, 
Hieber and Saal \cite{MR3916775}, Solonnikov~\cite{MR1855442}.  
We recall an interior estimate for the Stokes equation, note that in this case there arises no difficulty from mixed boundary conditions. We set 
\be\label{def:F_aD_a}
\begin{aligned}
F^a &:= - \nu \nabla((-A[\ub] + \id )\nabla \zb_{w,a}) - \nu [\nabla((-A[\ub] + \id)\nabla), \kappa]\zb_{w} \\
&\quad - \rred{\kappa \bigg(\zb_{w} ((K[\ub])\nabla)\wh + \wh ((K[\ub])\nabla)\zb_{w} } + (K[\ub] - \id)\nabla \zb_p\bigg)\\
&\quad + [\kappa, \nu \nabla^2_x]z_w + [\kappa, \nabla]z_p + \kappa f,\\
F^a_2 &:= \dddiv_{\id - K[\ub]^{\top}}\zb_{w,a} + [\kappa, \dddiv]z_w + [\dddiv_{\id - K[\ub]}, \kappa]z_w + \kappa f_2.
\end{aligned}
\ee
\begin{lemma}
Choosing $p=p'>2$  
we have
 \be\label{Stokes-Lp}
 \norm{\kappa z_w}{W^{2,p}(\Om_2)} + \norm{\kappa z_p}{W^{1,p}(\Om_2)} \le c\norm{F_a}{L^p(\Om_2)}  +\norm{D_a}{W^{1,p}(\Om_2)}.
\ee
 \end{lemma}
 \begin{proof}
 For a proof see \cite[Thm 11.3.4]{MR2563641}; we use the fact that $\kappa w \in W^{1,2}_0(\Om_{\kappa})$.
 \end{proof}

\begin{lemma} \label{lem:w_b} 
Let Hypothesis \ref{hyp2} be satisfied. Then, we have for solution $(z_w,z_p)$ of \eqref{eq:25} for $\varepsilon>0$ 
\be\label{z-w-a}
\begin{aligned}
&\norm{z_{w,a}}{W^{2,p}(\Om_2)} + \norm{z_{p,a}}{W^{1,p}(\Om_2)} 
 \le c\norm{(\kappa f,\kappa f_2)}{ L^p(\Om_{\kappa})  \times W^{1,p}(\Om_{\kappa}) } \\
 &\quad +c\norm{( f,f_2,g,f_3)}{\cals^2}
  + c r_1 \norm{\zb_{w,a}}{W^{2,p}(\Om_2)}  + cr_{2} \norm{\zb_{w}}{W^{2,2}(\Om_2)} \\&\quad 
 + cr_1^s\left(\norm{\zb_{p,a}}{W^{1,p}(\Om_2)}+ \norm{\zb_{p}}{W^{1,2}(\Om_2)}
\right) + \varepsilon \left(\norm{z_w}{W^{2,2}(\Om_2)} + \norm{z_p}{W^{1,2}(\Om_2)} \right) \\
& \quad + c_{\varepsilon} \left(\norm{z_w}{W^{1,2}(\Om_2)} + \norm{z_p}{L^2(\Om_2)} \right).
\end{aligned}
\ee
\end{lemma}
\begin{proof}
%
(i) We start with  \eqref{Stokes-Lp}. 
 Recalling ideas from \cite{MR3825153}, to estimate the norm $\norm{\kappa (K[u] - \id)\nabla z_p}{L^p(\Om_2)} $ we cannot use an estimate as \eqref{abcd1} in a  higher $L^p$--norm, since we have no $W^{1,p}(\Om_2)$ regularity of the pressure up to the boundary. Hence, we use the property of the communtator that
\be
\kappa (K[u] - \id)\nabla z_p = (K[u] - \id)\nabla z_{p,a} + (K[u] - \id)[\nabla, \kappa] z_p
\ee
and that the commutator looses one derivative 
\be
\begin{aligned}
\rred{\left[\nabla, \kappa \right] z_p} &\rred{= \nabla (\kappa z_p)  - \kappa \nabla z_p = (\nabla \kappa) z_p + \kappa \nabla z_p - \kappa \nabla z_p   =  z_p \nabla \kappa}
\end{aligned}
\ee
implying that
\be
\begin{aligned}
\norm{\kappa (K[u] - \id )\nabla z_p}{L^p(\Om_2)}
& \le c\norm{K(u) - \id}{L^{\infty}(\Om_2)}\norm{z_{p,a}}{W^{1,p}(\Om_2)} \\
&\quad + c(\kappa)\norm{K[u] - \id}{L^{\infty}(\Om_2)}\norm{z_p}{L^p(\Om_2)}\\
&\le c \norm{K[u] - \id}{L^{\infty}(\Om_2)}\norm{z_{p,a}}{W^{1,p}(\Om_2)}\\
 &\quad + c\norm{K[u] - \id}{L^{\infty}(\Om_2)}\norm{z_p}{W^{1,2}(\Om_2)},
\end{aligned}
\ee
using the continuous embedding $W^{1,2}(\Om_2) \subset L^p(\Om_2)$, $p < \infty$. 
This term can then be estimated as in \eqref{abcd1}.

(ii) Using estimates from the proof of Lemma \ref{lem:w_a}, estimates for the commutator, and the consideration from~(i) we obtain 
\be\label{first-term-est-p}
\begin{aligned}
\norm{F^a}{L^p(\Om_2)} & \le \cyan{cr_1^s 
\norm{\zb_{w,a}}{W^{2,p}(\Om_2)}} + \rred{c r_{2}}
\rred{\norm{\zb_w}{W^{2,2}(\Om_2)}}\\
&\quad + cr_1^s
\left(\norm{\zb_{p,a}}{W^{1,p}(\Om_2)} + \norm{\zb_p}{W^{1,2}(\Om_2)}\right)\\
&\quad + \brown{c\norm{z_w}{W^{1,p}(\Om_2)} + c\norm{z_p}{L^p(\Om_2)}}+ c\norm{\kappa f}{L^p(\Om_{\kappa})},\\
\norm{F^a_2}{W^{1,p}(\Om_2)} 
& \le \cyan{cr_1}
\norm{\zb_{w,a}}{W^{2,p}(\Om_2)} + 
\orange{c \norm{\zb_w}{L^p(\Om_2)}}+ c\norm{\kappa f_2}{ W^{1,p}(\Om_{\kappa})}.
\end{aligned}
\ee
For the terms $\norm{z_w}{W^{1,p}(\Om_2)} + c\norm{z_p}{L^p(\Om_2)}$ we cannot apply \eqref{est-lower} directly for $p>2$.
Using Ehrling's lemma we have for  $\varepsilon > 0$
\be
\norm{z_w}{W^{1,p}(\Om_2)}  \le \varepsilon \norm{z_w}{W^{2,2}(\Om_2)} + c_{\varepsilon} \norm{z_w}{W^{1,2}(\Om_2)}
\ee
which yields  
\be
\begin{aligned}
\brown{\norm{z_w}{W^{1,p}(\Om_2)} + c\norm{z_p}{L^p(\Om_2)}} &\le \varepsilon(\norm{z_w}{W^{2,2}(\Om_2)} + \norm{z_p}{W^{1,2}(\Om_2)} ) \\
&\quad + c_{\varepsilon} (\norm{z_w}{W^{1,2}(\Om_2)} + \norm{p}{L^2(\Om_2)} )
\end{aligned}
\ee
which allows to sublimate the higher order terms and gives, with $\varepsilon$ arbitrarily small, the result.
\end{proof}

    \subsection{Limit behaviour}\label{sec:limit}
The map \eqref{map-T} defines an iteration scheme generating a sequence of iterates 
\be
  (z_{w,n},z_{p,n})=(z_{w,a},z_{p,a}) + (z_{w,b},z_{p,b}) \in W^p
  \ee
  We will verify that it converges for $n\rightarrow \infty$ towards the unique solution $(z_w,z_p)\in W^p$ of \eqref{equ1-2q}. For a  $\eta \in ]0,1[$ we will estimate 
 \be
 \begin{aligned}
\norm{z_{w,n+1} - z_{w,n}}{\Ww} + \norm{z_{p,n+1} - z_{p,n}}{\Wpp} 
& \le \eta \bigg( \norm{z_{w,n} - z_{w,n-1}}{\Ww}\\
&\quad + \norm{z_{p,n} - z_{p,n-1}}{\Wpp}\bigg)
\end{aligned}
\ee 
for an arbitrary compact subset $\widehat{\Om}_2 \subset \Om_2$. 
Then,  there exists $(\zb_w,\zb_p)\in \Wp$ and sequence $((z_{w,n},z_{p,n}))_{n\in \mathds{N}} \subset \Wp$   
such that
\be
\begin{aligned}
z_{w,n} &\rar \zb_w \text{ in } \Ww && \text{as } n \rar + \infty,\\
z_{p,n} &\rar \zb_p \text{ in } \Wpp&& \text{as } n \rar + \infty.
\end{aligned}
\ee
with $(\zb_w,\zb_p)$ the unique solution of \eqref{equ1-2q}.   
This idea is taken from Grandmont~\cite{MR1891075}.

    Next, we show the strategy in detail. 

\subsubsection{The linearized state equation: Contraction property}
Let $Y_1 := (z_w^1, z_p^1)$, $Y_2 = (z_w^2, z_p^2)$, and $\bar Y_i := (\zb_w^i, \zb_p^i)$, $i=1,2$, with 
\be
Y_1 = T\Yb_1,\quad Y_2 = T\Yb_2.
\ee
Our aim is to show that 
\be
\norm{Y_1 - Y_2}{\Wp} = \norm{T(\Yb_1 - \Yb_2)}{\Wp} \le \eta \norm{\Yb_1 - \Yb_2}{\Wp}
\ee
where $\eta < 1$ uniform in $\widehat{\Om}_2$. From the definition of the map $T$ we write
\be
\left\{
\begin{aligned}
-\nu \nabla(\nabla z_w^i) + \nabla z_p^i &= - \nu \nabla((-A[u] + \id)\nabla \zb_w^i)\\
& \quad \rred{- \wh((K[u])\nabla)\zb_w^i  - \zb_w^i((K[u])\nabla)\wh}\\
&\quad - (K[u]- \id)\nabla\zb_p^i +f =: D(\Yb_i)\quad \text{in } \Om_2\\
\dddiv z_w^i &= \dddiv_{\id - K^{\top}(u)}\zb_w^i+f_2 =: B(\Yb_i)\quad \text{in } \Om_2\\
z_w^i &= g\quad \text{on } \Gammain\\
-\nu \partial_n z_w^i+ z_p^i n &= 
-\partial_{A[u] - \id,n }\zb_w^i + \zb_p^i(K[u] - \id) \cdot n + f_3\quad \text{ on } \Gammaout
\end{aligned}
\right.
\ee
for $i=1,2$.
\if{and
\be
\left\{
\begin{aligned}
-\nu \nabla (\nabla z_{w,2}) + \nabla z_{p,2} &= \nu \nabla((-A[u] + \id)\nabla \wb_2)\\
& \quad \rred{- \wh((-K[u])\nabla)\zb_{w,2}  +\zb_{w,2}((-K[u])\nabla)\wh}\\
& \quad- (K[u] - \id)\nabla\pb_2  =: D(\Yb_2)\\
\dddiv z_{w,2} &= \dddiv_{\id - K[u]^{\top}}\zb_{w,2} = B(\Yb_2)\\
z_{w,2} &= g\quad \text{ on } \Gammain \\
- \nu \partial_n z_{w,2}
+ z_{p,2}n &= - \partial_{A[u]-id,n} \wb_2 + \pb_2(K[u] - \id) \cdot n\quad \text{ on } \Gammaout.
\end{aligned}
\right.
\ee
}\fi
Denoting $\Yb := \Yb_1 - \Yb_2$ we obtain the equation for 
\be
Y := Y_1 - Y_2 =: (Z_{w} , Z_{p})
\ee
in terms of $\Yb_i \in B_r(W^p)$:
\be\label{equ:ref}
\left\{
\begin{aligned}
- \nu \nabla(\nabla Z_w ) + \nabla Z_p &= D(\Yb_1) - D(\Yb_2)&&\text{in } \Om_2,\\
\dddiv Z_w &= B(\Yb_1) - B(\Yb_2) = \dddiv_{\id-K^{\top} (u)}\bar{Z}_{w} &&\text{in } \Om_2,\\
Z_w &= 0 &&\text{ on }\Gammaint \cup \Gammain,\\
-\nu \partial_n Z_w+ Z_p n & = - \partial_{A(u)-\id,n} \bar{Z}_w 
+ \bar{Z}_p(K[u] - \id) \cdot \nx && \text{ on } \Gammaout.
\end{aligned}
\right.
\ee

\begin{lemma}
\label{lem:1001}
 Let Hypothesis \ref{hyp2} be satisfied. For the solution 
 of \eqref{equ:ref} we have 
\be
\norm{Z_w}{W^{2,2}(\Om_2)} + \norm{Z_p}{W^{1,2}(\Om_2)} \le c(r_1^
s + r_1 + r_{2})(\norm{\bar{Z}_w}{W^{2,2}(\Om_2)} + \norm{\bar{Z}_p}{W^{1,2}(\Om_2)})
\ee 
where $\bar{Z}_w := \zb_w^1 - \zb_w^2$ and
$\bar{Z}_p := \zb_p^1 - \zb_p^2$.
\end{lemma}
\begin{proof}
 We proceed similarly as in Lemma \ref{lem:w_a} using also Theorem \ref{42}; we estimate
 \be\label{estimate:D}
\begin{aligned}
\norm{D(Y_1) - D(Y_2)}{L^2(\Om_2)} & = \norm{\nu \nabla ((-A[u] + \id)\nabla \bar{Z}_w}{L^2(\Om_2)}  \\
&\quad \rred{+ \norm{\wh ((- K[u])\nabla)\bar{Z}_w - \bar{Z}_w(K[u]\nabla)\wh}{L^2(\Om_2)}}\\
&\quad +\norm{ (K[u] - \id)\nabla\bar{Z}_p }{L^2(\Om_2)}\\
&\le c (r_1+r_1^s+r_{2}) (\norm{\bar{Z}_w}{W^{2,2}(\Om_2)} +\norm{\bar{Z}_p}{W^{2,2}(\Om_2)}).
\end{aligned}
\ee
For the term $B(Y_i)$  we have by \eqref{est:1-p} that with $1/p_1 + 1/p_2 =1/2$, $p_1>n$, 
\be\label{B-est}
\begin{aligned}
&\norm{B(\Yb_1) - B(\Yb_2)}{W^{1,2}(\Om_2)}  = \norm{\dddiv_{\id - K[u]^{\top}}\bar{Z}_w }{W^{1,2}(\Om_2)}\\
&\quad \le  \norm{\id - K[u]}{W^{1,p_1}(\Om_2)}
\norm{\bar{Z}_w}{W^{1,p_2}(\Om_2)}+  \norm{\id - K[u]}{L^{\infty}(\Om_2)}
\norm{\bar{Z}_w}{W^{2,2}(\Om_2)}\\
&\quad \le  \norm{\id - K[u]}{W^{1,p_1}(\Om_2)}
\norm{\bar{Z}_w}{W^{2,2}(\Om_2)}+  \norm{\id - K[u]}{L^{\infty}(\Om_2)}
\norm{\bar{Z}_w}{W^{2,2}(\Om_2)}\\
&\quad \le cr_1^s\norm{\bar{Z}_w}{W^{2,2}(\Om_2)}
\end{aligned}
\ee
and on the boundary $\Gammaout$
\begin{align}\label{N1}
- \nu  \partial_n Z_w + Z_p \cdot \nx & = - \partial_{A[u] - \id,n} Z_w
+ \bar{Z}_p K[u]\cdot \nx - \bar{Z}_p\cdot \nx. 
\end{align}
From Lemma \ref{lem:w_a} it follows that
\begin{multline}
\norm{Z_w}{W^{2,2}(\Om_2)} + \norm{Z_p}{W^{1,2}(\Om_2)}  \le c \norm{D(\Yb_1) - D(\Yb_2)}{L^2(\Om_2)} \\ + c\norm{B(\Yb_1) - B(\Yb_2)}{W^{1,2}(\Om_2)} + c\norm{- \partial_{A[u]- \id,n} \bar{Z}_w}{W^{1/2,2}(\Gammaout)} \\
 + c\norm{\bar{Z}_p K[\ub]  \cdot \nx}{W^{1/2,2}(\Gammaout)}   + c\norm{\bar{Z}_p  \cdot \nx}{W^{1/2,2}(\Gammaout)}.
 \end{multline}
Using estimates \eqref{eq:lem-last} and \eqref{est-bdry} for the boundary terms we conclude.
 \end{proof}

\begin{lemma}\label{lem:contraction} Let Hypothesis \ref{hyp2} be satisfied and additionally, $r_1>0$ and $r_2>0$ be sufficiently small. Then, the map $T$ defined by \eqref{map-T} satisfies for some $0<\eta<1$  
\be\label{T-contr}
\norm{T(Y_1 - Y_2)}{\Wp} < \eta \norm{\bar{Y}_1 - \bar{Y}_2}{\Wp}
\ee
for $\widehat{\Om}_2 \subset \Om_2$ compact.
\end{lemma}
\begin{proof}As a consequence of the previous lemma it remains to prove the contraction property with respect to higer $p$-integrability on compact subsets. 

We recall the function $\kappa$.
 We remark that the commutator has for sufficiently smooth $v$ the property that
\be
\begin{aligned}  \text{ }
[ \kappa  , D_x]v & = -D_x (\kappa v) + \kappa D_x v,\quad 
[\kappa,D_x^2] v =-D^2_x (\kappa v) + \kappa D^2_x v.
\end{aligned}
\ee
Hence, we have
\be
\begin{aligned}
- \nu \nabla(\nabla Z_{w,a}) + \nabla Z_{p,a} & = \kappa(D(\Yb_1) - D(\Yb_2)) + [\kappa, \nu D^2_x
] Z_{w} \\
&\quad + [\kappa, \nabla] Z_p && \text{in } \Om_2,\\
\dddiv Z_{w,a} & = \kappa(B(\Yb_1) - B(\Yb_2)) + [\kappa, \dddiv] Z_w  && \text{in } \Om_2 \\
 Z_{w,a} & = 0 && \text{on }\Gammaint \cup \Gammain \cup \Gammaout.
\end{aligned}
\ee
with $Z_{w,a}:=z_{w,a}^1 - z_{w,a}^2$ and $Z_{p,a}:=z_{p,a}^1 - z_{p,a}^2$.
Since the commutators loose one order of derivative we can derive higher Lebesgue integrability, i.e. for $(w,p) \in W^{2,2}(\Om_2) \times W^{1,2}(\Om_2)$ 
\be\label{est:1001}
\begin{aligned}
\norm{[\kappa, D^2_x]w}{L^p(\Om_2)} &\le C\norm{w}{W^{1,p}(\Om_2)} \le c\norm{w}{W^{2,2}(\Om_2)},\\
\norm{[\kappa, \dddiv]w}{W^{1,p}(\Om_2)} &\le C\norm{w}{W^{1,p}(\Om_2)} \le c\norm{w}{W^{2,2}(\Om_2)},\\
\norm{[\kappa, \nabla]p}{L^{p}(\Om_2)} &\le C\norm{p}{W^{1,2}(\Om_2)}.
\end{aligned}
\ee
Similar as in the proof of Lemma \ref{lem:w_b} we estimate $\norm{\kappa(B(\Yb_1) - B(\Yb_2))}{W^{1,p}(\Om_2)}$ and 
$\norm{\kappa(D(\Yb_1) - D(\Yb_2))}{W^{1,p}(\Om_2)}$. 
\if{
\be
\begin{aligned}
\norm{\kappa(B(\Yb_1) - B(\Yb_2))}{W^{1,p}(\Om_2)} &\le c r_1 \norm{\Zb_{w,a}}{W^{2,p}(\Om_2)},\\
\norm{\kappa(D(\Yb_1) - D(\Yb_2))}{W^{1,p}(\Om_2)} &\le c (r_1 + r_1^s + r_{2}) \norm{\Yb_1 - \Yb_2}{\Wp}.
\end{aligned}
\ee}\fi
Here we use the same trick as in that proof to obtain higher $p$-integrability, namely we switched around the order of $\kappa$ and the differential operators in the term with coefficient $A[u]$ as well as in the divergence term and introduce a commutator as correction term.

Applying further the estimate of Lemma \ref{lem:1001} to the terms \eqref{est:1001} we obtain finally
\be
\begin{aligned}
\norm{Z_{w,a}}{W^{2,p}(\Om_2)} + \norm{Z_{p,a}}{W^{1,p}(\Om_2)} &\le 
c (r_1 + r_1^s + r_{2} )\norm{\Yb}{\Wpk}.
\end{aligned}
\ee
Thus, for $r_1>0$ and $r_2>0$ sufficiently small we obtain the result.

\end{proof}    
    \begin{theorem}\label{thm:lin-equ-higher-reg}
 Let Hypothesis \ref{hyp2} be satisifed and additionally $r_1>0$ and $r_2>0$ sufficiently small. For data satisfying the regularity assumption in \eqref{data-high}, $\wh \in
 B_{r_2}(\widehat{W}^{2,p}(\Om_2)) $, and $u \in B_{r_1}(U^p)$   
 the linearized equation \eqref{system-rhs-zero-1} has a unique solution $(z_w,z_p) \in W^p. $
 Moreover, the solution is bounded by the data, we have
 \be\label{lin:est}
 \begin{aligned}
   \norm{(z_w,z_p)}{\Wp}  &\le c \norm{f}{L^p(\widehat{\Om}_2)\cap L^2(\Om_2)} + c \norm{f_2}{W^{1,p}(\widehat{\Om}_2)\cap W^{1,2}(\Om_2)} \\
  &\quad + c\norm{\delta g}{W^{3/2,2}(\Gammaint)} +c\norm{f_3}{W^{1/2,2}(\Gammaint)}
 \end{aligned}
\ee
for compact subsets $\widehat{\Om}_2 \subset \Om_2$.
\end{theorem}
\begin{proof}
 The existence follows by the procedure described at the beginning of Section~\ref{sec:limit} and 
  the contraction property given in Lemma~\ref{lem:contraction}. The estimate follows from the boundedness of the operator $T$ shown in Lemma \ref{lem:w_a} and \ref{lem:w_b} and sublimating the with powers of $r_i$ weighted terms by the left hand side.
\end{proof}

\begin{hypothesis}\label{hyp4}
 Let $r_1>0$ and $r_2>0$ be sufficiently small, such that for $\wh \in \widehat{W}^{2,p}(\Om_2)$ and $u \in B_{r_1}(U^p)$  the linearized equation \eqref{system-rhs-zero-1} has a unique solution in $W^p$ satsfying estimate \eqref{lin:est}.
\end{hypothesis}
    
 \section{Differentiability}\label{sec:diff}
 
 In this section we show the main result, the differentiability of the mapping which maps the infow profile $g$ to the velocity-pressure-deformation triple $(u,w,p)$ of the fluid-structure interation system. We follow in parts ideas from \cite{MR3959888} where linear elasticity is coupled with the Stokes equation  with Dirichlet boundary conditions in a smooth domain. 
  In a first step we consider the differentiability of the data-to-solution map $g$ to $(w,p)$ for the Navier-Stokes system and in particular of the traction operator $\tau$.
  
We introduce two systems, which will appear to be the linearized systems with respect to inflow data $g$ and with respect to perturbation $u$, namely 
  \begin{equation}
\left\{
  \begin{aligned}\label{state-lin1-2}
   -\nu \nabla(A[u] \nabla \delta w_g) +  \delta w_g (K[u]\nabla)\wh  & +  \wh (K[u]\nabla)\delta w_g \\
   + K[u]\nabla \delta p_g&=0 && \text{in }\Om_2,\\
   \operatorname{div}_{K[u]^{\top}} \delta w_g&=0&&\text{in }\Om_2,\\
\delta w_g&= \delta g &&\text{on }\Gammain,\\
\delta w_g&=0&&\text{on }\Gammawall \cup \Gammaint,\\
- \nu \partial_{A[u],n} \delta w_g + \delta p_gK[u]\cdot \nx&=0&&\text{on } \Gamma_{\text{out}}\\
   \end{aligned}
   \right.
   \end{equation}
and
\begin{equation}
   \left\{
  \begin{aligned}
\label{equa-5}
   -\nu \nabla(A[\uh] \nabla \delta w_u)  &\cyan{+ \delta w_u K[\uh] \nabla \wh }
     + \cyan{\wh K[\uh] \nabla \delta w_u}
    - K[\uh] \nabla \delta p\\ & =\cyan{- \wh K'[\uh] \delta u \nabla \wh} +\gamma\nu \nabla(A'(\uh) \delta u \nabla  \wh)  \\
    &\quad  -  K'[\uh] \delta u \nabla \ph && \text{in }\Om_2,\\
   \operatorname{div}_{K^{\top}(\uh)}\delta w_u&= \operatorname{div}_{K^{\top}(\uh)\delta u} \wh &&\text{in }\Om_2,\\
w&= 0 &&\text{on }\Gammain,\\
w&=0&&\text{on }\Gammawall \cup \Gammaint,\\
- \nu \partial_{A[\uh]} \delta w_u + \delta pK[\uh] \cdot \nx&= \nu \partial_{A'[\uh]\delta u,n} w - pK'[\uh] \delta u \cdot \nx &&\text{on } \Gamma_{\text{out}}.
   \end{aligned}
   \right.
   \end{equation}
  For given $(\uh,  \gh )\in B_{r_1}(U^p)\times B_r(\calg_{3/2})$ 
  we write the Navier-Stokes equation~\eqref{equ:10010} as  
 \be\label{e-equation}
 e\colon X^p \times \calg_{3/2} \rar \cals^{p'},\quad e(u,w,p,g ) = 0,
 \ee
\if{with $B$ given by \eqref{ }.
\be
\begin{aligned}
\dddiv(B w_f) & =f_2 && \text{in }\Om_2,\\
w_f &= 0 && \text{on } \Gamma_i
\end{aligned}
\ee
}\fi
with
 \be\label{e2}
 e(u,w,p,g): = 
 \left( \begin{array}{l}
 - \nu \nabla(A[u]\nabla w ) + w (K[u]\nabla)w + K[u]\nabla p \\
 (K[u]\nabla)^{\top} w\\
 w|_{\Gammain} -g  \\
- \nu (A[u]\nabla w) \cdot \nx + pK[u] \cdot \nx
\end{array}
\right).
 \ee
 \begin{lemma}
 The function $e$ defined in \eqref{e-equation}--\eqref{e2} is continuously differentiable.
 \end{lemma}
 \begin{proof}
 The statement follows by the regularity of the appearing functions and the smoothness of $A$ and $K$, see Lemma \ref{grandmont}.
 \end{proof}
 To apply the implicit function theorem we show that the derivative of $e$ with respect to $(w,p)$ defines an isomorphism in a solution $(\uh,\wh,\ph,\gh)$ of \eqref{e-equation}.
 
 Let  $\uh \in W^{2,p}(\Om_1)$ and $(\wh,\ph)\in W^p$ the corresponding solution of the Navier-Stokes equation \eqref{system-rhs-zero}. Moreover, let $(F,F_2,g,F_3) \in \cals^{p'}$. 
  Recalling Hypothesis \ref{hyp1} and \ref{hyp2}, we consider the solution $(z_w,z_p) \in W^p$ of 
  \be
  D_{(w,p)} e(\uh,\wh,\ph,\gh)(z_w,z_p)=(F,F^2,g,F^3)^{\top}.
  \ee
  By Theorem \ref{thm:lin-equ-higher-reg} the solution is well-defined and we have
  \if{
  \be
  \norm{(z_w,z_p)}{W^p} \le c (\norm{F}{\widehat{W}^{0,p}(\Om_2)} +\norm{F_2}{W^{1,p}(\Om_2)}+\norm{g}{\calg_{3/2}}+\norm{F_3}{W^{1/2,2}(\Om_2)}).
  \ee
  }\fi
  \be\label{est-101}
 \begin{aligned}
  & \norm{(z_w,z_p)}{W^{2,p}(\widehat{\Om}_2) \cap W^{2,2}(\Om_2) \times W^{1,p}(\widehat{\Om}_2)\cap W^{1,2}(\Om_2)}  \le c \norm{F}{L^p(\widehat{\Om}_2)\cap L^2(\Om_2)} \\
  &\quad + c \norm{F_2}{W^{1,p}(\widehat{\Om}_2)\cap W^{1,2}(\Om_2)}
  + c\norm{g}{W^{3/2,2}(\Gammaint)} +c\norm{F_3}{W^{1/2,2}(\Gammaint)}. 
 \end{aligned}
\ee
 
\if{ (i) The mapping
 \be
 \mathcal{G} \rar W^p,\quad g \mapsto (w[g],p[g])
 \ee
 is continuous.
 
 (ii) The mapping
 \be
 U^p \rar W,\quad u \mapsto (w[u],p[u])
 \ee
 is Lipschitz continuous.}\fi
\if{(i) \emph{Existence:} We follow closely ideas from \cite{MR3825153}. We consider the fixed point equation
\be
\caln\colon W^p \rar W^p,\quad (w,p)=\caln(\wb,\pb),
\ee
where $\caln$ maps for given $g\in \mathcal{G}$ the point $(\wb,\pb)$ to the solution $(w,p)$ of 
\be\label{equ:10010}
\left\{
\begin{aligned}
  -\nu \nabla ( \nabla w ) + \nabla p  & =  -\nu \nabla ( \nabla (A[u] - \id)\wb ) + (\wb (K[u] - \id) \nabla) \wb  \\
 & \quad + (K[u] - \id) \nabla \pb   && \text{in } \Om_2,\\
 -\dddiv w&=-((K[u] - \id) \nabla)^{\top} \wb &&\text{in }\Om_2,\\
 w &=g&&\text{on }\Gammain,\\
 w &=0&&\text{on }\Gamma_{\text{wall}} \cup \Gamma_{\text{int}},\\
 -\nu \partial_{n}  w + p n&=-\nu (\partial_{A[u],n} -\partial_{n}) \wb + \pb (K[u] - \id) n  &&\text{on } \Gamma_{\text{out}}.
\end{aligned}
\right.
\ee
Existence follows by Banach's fixed point theorem, see \cite[Lem. 4.5]{MR3825153}.
}\fi
\if{Let $(dat_k)_{k\in \NNN}:=((F_k,F^2_k,g_k,F^3_k))_{k\in \NNN} \subset \cals^{p'}$ be a sequence with $dat_k \rar dat \in \cals^{p'}$. And let $\norm{\wb}{W_w}>0$ be sufficiently small such that $T_{dat}$ and $T_{dat_k}$ have unique fixed points $(w_g,p_g)$ and $(w_k,p_k)$ in $B_r(W^p)$.
Since using estimates \eqref{W2-est} and \eqref{Stokes-Lp}   
\be
\begin{aligned}
\norm{T_{dat}((\wb,\pb)) - T_{dat_k}((\wb,\pb))}{W^p}& \le c \norm{F - F_k}{L^p(\Om_2)} + \norm{F_2 - F^2_k}{W^{1,p}(\Om_2)}\\
&\quad  + \norm{g - g_k}{W^{3/2,2}(\Gammain)} + \norm{F^3 - F^3_k}{W^{1/2,2}(\Gammaout)} \\
& =:\eta_k \quad \text{uniform in } (\wb,\pb) \in W
\end{aligned}
\ee
for any $\eta_k >0$, there exists a $k_0\in \NNN$ such that with contraction constant $0<\eta <1$ and for all $k\ge k_0$ we have
\be
\begin{aligned}
\norm{(w_k,p_k) - (w_{dat},p_{dat})}{W} &= \norm{T_{{dat}}(w_{dat},p_{dat}) -  T_{{dat}_k}(w_k,p_k)}{W}\\
& \le  \norm{T_{{dat}}(w_{dat},p_{dat}) -T_{{dat}}(w_k,p_k) }{W} + \eta_k \\
& \le \eta \norm{(w_{dat},p_{dat}) -(w_k,p_k) }{W}+ \eta_k.
\end{aligned}
\ee
Hence, we have
\be
\begin{aligned}
\norm{(w_k,p_k) - (w_g,p_g)}{W} &\le \eta_k / (1-\eta)
\end{aligned}
\ee
showing the estimate.
\end{proof} }\fi

 \begin{lemma}\label{lem:999}
Let Hypothesis \ref{hyp1} and \ref{hyp4} hold and $\widehat{\Om}_2 \subset \Om_2$ compact.

 (ia) The mapping $\caln\colon B_{r_1}(U^p)\times B_{r_2}(\calg_{3/2}) \rar \Wp$ with $(u,g)\mapsto (w[u,g],p[u,g])$ is continuously differentiable.

    (ib) Let $u\in B_{r_1}(U^p)$ be fixed. The derivative  $(\delta w_g,\delta p_g)$ of
    \be
    B_r(\calg_{3/2}) \rar \Wp,\quad g\mapsto (w[g],p[g])
    \ee
     is given by \eqref{state-lin1-2}.
    
    (ic)  Let $g\in B_r(\calg_{3/2})$ be fixed.  The derivative $(\delta w_u,\delta p_u)$ of
    \be
    B_{r_1}(U^p) \rar \Wp,\quad u\mapsto (w[u],p[u])
    \ee
      is given by \eqref{equa-5}.

 (ii) The mapping 
\begin{equation}\label{operator-F}
\begin{aligned}
&\calf \colon B_{r_1}(U^p) \times B_r(\mathcal{G}_{3/2}) \rar W^{1-1/p,p}(\Gammaint),\\
& (u ,  g ) \mapsto t(u,p)=p[u]K[u]\cdot \nx
\end{aligned}
\end{equation}
is continuously differentiable. 
 \end{lemma}
\begin{proof}
 (ia)  
To show continuous differentiability of $(w[\cdot],p[\cdot])$, we employ the implicit function theorem. 
We note that 
\be
D_{(w,p)} e(u,w,p,g)\colon  \Wp  \rar  \cals^{p'}, 
\ee
corresponds to the transformed Stokes operator on the left given by
\be
\begin{aligned}
\left( \begin{array}{l}
 -\nu \nabla(A[u]\nabla \delta w ) + \delta w (K[u]\nabla)\wh + \wh (K[u]\nabla)\delta w + K[u]\nabla \delta p \\
 (K[u]\nabla)^{\top} \delta w\\
\delta  w|_{\Gammain} \\
-\nu (A[u]\nabla \delta w) \cdot \nx + \delta p K[u] \cdot \nx
\end{array}
\right).
 \end{aligned}
\ee
We observe that $ D_{(w,p)} e(u,w,p,g) \colon \Wp \rar \cals^{p'}$ is an isomorphism by Theorem~\ref{thm:lin-equ-higher-reg} and estimate given there, cf. \eqref{est-101}.

\if{
\begin{proposition}\label{prop:ex-diff-lin}
Let $\norm{\wh}{W^{1,2}(\Om_2)}$ be sufficiently small. Then the in $\uh$ linearized state equation \eqref{state-lin} has a unique solution\footnotetext{W use $W^{2,p}$ regularity. What do we really need?}
 \be
 (\delta w[\delta u],\delta p[\delta u]) \in W^{2,2}(\Om_2) \times W^{1,2}(\Om_2).
 \ee
\end{proposition}
\begin{proof}
The terms on the right hand side, more precisely $\nabla (A'(u)\delta u \nabla  \delta w)$, $\wh (K'(\uh)\delta u )^{\top} \nabla) \wh$, and $ K'(\uh)\delta u \nabla  \delta p$ are in $ L^p(\Om_2)$.
\if{
By applying \eqref{lem:rel} we transform equation \eqref{state-lin} to the physical domain and obtain\footnotetext{Beachte: Ein $J$ kommt als Faktor automatisch furch den Transformationssatz!}
\if{\begin{equation}
   \left\{
  \begin{aligned}
\label{equa-3}
 -\nu \nabla( \nabla \delta w) 
  & + \delta w  \nabla \wh  + \wh \nabla \delta w + \nabla \delta p = \nu (J K^{-1} \nabla) (A'(\uh) \delta u (J K^{-1} \nabla)  \wh )J  \\
   & \quad + K'(\uh) \delta u  (J K(\uh)^{-1}\nabla) \ph J\\
    & \quad + \wh K'(\uh) \delta u (J K(\uh)^{-1} \nabla) \wh J  && \text{in }\Om_2,\\
    & \quad +(w^{\top} K'[u]\delta u K[u] \nabla\wh\\
 \dddiv w&=0&&\text{in }\Om_2,\\
w&=g&&\text{on }\Gamma,\\
w&=0&&\text{on }\Gamma_{\text{wall}} \cup \Gamma_{\text{int}}[\uh]\\
\nu \frac{\partial w}{\partial n} - p \nx &=0&&\text{on } \Gamma_{\text{out}}.
   \end{aligned}
   \right.
   \end{equation}
}\fi

   \begin{equation}
   \left\{
  \begin{aligned}
\label{equa-3-b}
 -\nu \nabla( \nabla \delta w) &
   + \delta w  \nabla \wh  + \wh \nabla \delta w + \nabla \delta p \\
   &= \nu (G[\uh]^{-1} \nabla) (A'(\uh) \delta u (G[\uh]^{-1} \nabla)  \wh )J  \\
   & \quad + K'(\uh) \delta u  (G[\uh]^{-1}\nabla) \ph J\\
    & \quad + \wh K'(\uh) \delta u (G[\uh]^{-1} \nabla) \wh J  && \text{in }\Om_2,\\
  \dddiv \delta  w&= - \dddiv( K'[u]\delta u w)&&\text{in }\Om_2,\\
w&=0 &&\text{on }\Gamma,\\
w&=0&&\text{on }\Gamma_{\text{wall}} \cup \Gamma_{\text{int}}(\uh),\\
\nu \partial_{\nx} w  - p \nx &=0&&\text{on } \Gamma_{\text{out}}.
   \end{aligned}
   \right.
   \end{equation}
We check that the linearized equation \eqref{equa-3-b} has a unique solution.
 We check the boundedness of the terms on the right hand side:
  For $(w,p) \in W^{2,2}(\Om_2) \times W^{1,2}(\Om_2)$ we have 
using the algebra property in $W^{1,p}(\tilde{\Om}_2)$
\begin{align}
 \nu (G[u]^{-1} \nabla) (A'(\uh) \delta u (G[u]^{-1} \nabla)  \wh )J  &\in L^p(\tilde \Om_2), \\
 K'(\uh) \delta u  J K(\uh)^{-1}\nabla  \ph &\in L^p(\tilde \Om_2), \\
 \wh K'(\uh) \delta u J K(\uh)^{-1} \nabla  \wh &\in L^p(\tilde \Om_2).
\end{align}
}\fi
Thus, applying Theorem \ref{thm:ex-lin} we have the result.
    \end{proof}
    }\fi

(ib) With $ D_{g} e(u,w,p,g)\delta g$ given by
\be
\begin{aligned}
\left( \begin{array}{llll}
 0, & 0,&  \delta g, & 0\\
\end{array}
\right)
 \end{aligned}^{\top}
\ee
the derivative $(\delta w_g,\delta p_g)$ with respect to $g$ is given as the solution of
    \be
    D_{(w,p)} e(u,w,p,g) (\delta w_g,\delta p_g)= - D_g e(u,w,p,g)\delta g
    \ee
or equivalently by \eqref{state-lin1-2}.
A solution exists by Theorem \ref{thm:lin-equ-higher-reg} and is bounded by the data, the result follows.  

(ic) Analogously, the partial derivative $D_{u} e(u,w,p,g)\delta u$ is given by
\be
 \left( \begin{array}{l}
 - \nu \nabla (A'[u]\delta u\nabla w ) + w (K'[u]\delta u\nabla)w + K'[u]\delta u \nabla p \\
 (K'[u] \delta u\nabla)^{\top} w\\
 0\\
- \nu (A'[u]\delta u\nabla w) \cdot \nx + pK'[u]\delta u  \cdot \nx
\end{array}
\right)
\ee
and \eqref{equa-5} can be written as
    \be\label{rhs-u}
    D_{(w,p)} e(u,w,p,g)(\delta w_u,\delta p_u)= - D_u e(u,w,p,g)
    \ee
    or equivalently by \eqref{equa-5}.
Since 
\be
\begin{aligned}
- \nu \nabla(A'(\uh) \delta u \nabla  \wh)  + \wh K'[\uh] \delta u \nabla \wh 
    +  K'[\uh] \delta u \nabla \ph & \in   \widehat{W}^{0,p}(\Om_2) \cap L^2(\Om_2), \\
    - \nu (A'[u]\delta u\nabla w) \cdot \nx + pK'[u]\delta  \cdot \nx & \in W^{1-1/p,2}(\Gammaint)
\end{aligned}    
\ee
for $p>2$, the right hand side in \eqref{rhs-u} has the suitable regularity and we conclude again with Theorem~\ref{thm:lin-equ-higher-reg}.

(ii) Follows directly from (ia). Note, that here we use that in the interior we have higher $p$-integrability and that $\Gammaint$ is bounded away from $\Gammaext$.
   \end{proof}

 \begin{lemma}\label{lem:shape} Let Hypothesis \ref{hyp1} and \ref{hyp4} be satisfied.
  For $g\in B_r(\calg_{3/2})$ and $u \in B_{r_1}(U^p)$ and $\calf$ given in \eqref{operator-F} we have for any $\varepsilon > 0$
\be
\norm{ \frac{\dd}{\dd u} \calf ( u , g )}{L_F} \le \varepsilon, 
\ee
with $L_F:=L(W^{2,p}(\Om_1),W^{1-1/p,p}(\Gammaint))$ provided that $r$ and $r_1$ are sufficiently small. 
 \end{lemma}
\begin{proof}\if{
\if{We note that $F$ is given as a composition of the following three maps:
in $\Om_f$,
with bounds on $R_1$ and $R_2$ given by \eqref{ } and \eqref{}. From the estimates \eqref{ } and \eqref{ } the assertion
\be
W_{i,2} \rar W_E, \quad \uh_s \mapsto E (\uh_s ) = \uh_f ,
\ee
\be
( \uh_f , \vh_f , \ph_f ) \mapsto \Jh \sigma f F - T n  .
\ee
\if{
Here, for fixed $\fh$ and $g$ the operator $G$ is defined as in Theorem \ref{ }.
\be
E \colon G :
\tau \colon  W_E \rar W_F,\quad 
\uh \mapsto (\vh ,\ph ) 
\ee
solving (F) with fixed data (fˆ ,g ,P􏰍),
\be
\dddiv(B e_v)=R_2\quad
e_v = 0 \text{ on } \Gamma_i \cup \Gamma_f, \quad W_E \times W_F \rar W_{i,1},
\ee
}\fi
}\fi
Let $\caln$ denote the solution operator for the Navier-Stokes operator as a function of $g$ and $u$ (see~Theorem~\ref{thm:ex-flow}) and recall that $\tau$ denotes the traction operator. }\fi
\if{\be
\begin{aligned}
 &\caln\colon  W^{2,p}(\Om_2) \times W^{3/2,2}(\Gammain) \rar W ,&&\text{(NSE solution operator)}\\
 &\tau \colon W^{2,p}(\Om_2) \times W^p \rar W^{1-1/p,p}(\Gammain); \quad (u_2 , w , p ) \mapsto p K[u] n, 
 &&\text{(traction mapping)}
\end{aligned}
\ee
}\fi
We write
\be
\calf(u, g ) = \tau(u, \caln(g,u)).
\ee
By Lemma \ref{lem:999} and applying the chain rule, we get for any direction $\delta u \in W^{2,p}(\Om_1)$ that
\be
\frac{\dd}{\dd u} \calf(u ,g)\delta u = \frac{\dd}{\dd u}\tau(u, \caln(g,u ))\delta u.
\ee
By Theorem \ref{thm:ex-flow} we can choose for $\delta>0$ the radii $r>0$ and $r_1>0$ sufficiently small such that $p \in B_{\delta}(W_p)$. 
 Using the smoothness of the outer normal on the interface taking into account that $\Gammaint$ is bounded away from $\Gammaext$ and recalling that $p>n$  we have 
\be\label{equ:rhs}
\begin{aligned}
\norm{\frac{\dd}{\dd u} \calf ( u , g )\delta u }{W^{1-1/p,p}(\Gammaint)}& \le \norm{z_{p,a}K[u]}{W^{1,p}(\widehat{\Om}_2)} + \norm{p_a[u]K'[u]\delta u}{W^{1,p}(\widehat{\Om}_2)} \\
& \le \norm{ z_{p,a} }{W^{1,p}(\widehat{\Om}_2)} \norm{K[u]}{W^{1,p}(\widehat{\Om}_2)}\\
&\quad + \norm{p_a}{W^{1,p}(\widehat{\Om}_2)} \norm{ K'[u]\delta u }{W^{1,p}(\widehat{\Om}_2)}
\end{aligned}
\ee
with $\widehat{\Om}_2\subset \Om_2$ a compact subset containing $\Om_1$. Note, that in \eqref{equ:rhs} we use higher $p$-integrability of $p_a$ whose support is bounded away from the boundary.
Now, using the estimate in Theorem \ref{thm:lin-equ-higher-reg} applied to \eqref{equa-5}, we have for any $\gamma>0$ and data sufficiently small 
that
\be
\begin{aligned}
\norm{\frac{\dd}{\dd u} \calf ( u , g )\delta u }{W^{1-1/p,p}(\Gammaint)}
& \le \gamma \norm{\delta u }{W^{2,p}(\widehat{\Om}_2)} & \le c \gamma \norm{\delta u }{W^{2-1/p,p}(\Gammaint)} \\
& \le c \gamma \norm{\delta u }{W^{2,p}(\Om_1)}
\end{aligned}
\ee
which shows the assertion. 
\if{
We take a closer look on the summands on the right hand side of \eqref{equ:rhs}.

(i) For the first term $D_u\tau$, we note
\be
\begin{aligned}
\norm{D_u \tau( \cald u , G ( u_f ) )(\delta u)}{} &= \norm{p K'[u]\delta u \cdot n_{\Om_1}}{} \le \norm{p}{???} \norm{K'[u]\delta u}{???}
\end{aligned}
\ee

(ii) Now, we come to the second term
\be
\frac{\norm{\tau [ u_f , G'( u_f ) \delta u ]}{}}{\norm{\delta u}{???}} \le 
\frac{
\norm{
\delta p[u] K[u] n_{\Om_1}
}{}
}
{\norm{\delta u}{}} \rar 0\quad (\norm{u}{}\rar 0).
\ee
We conclude.
}\fi
\end{proof}

 
 We state the main differentiability result on the mapping of the data to the solution of the fluid-structure interation problem.
 
 \begin{theorem}\label{thm:main} 
 Let  $r>0$ be sufficiently small.
 Then, the mapping
 \be
 \Pi \colon B_r(\calg_{3/2}) \rar X^p ,\quad  g \mapsto (u[g],w[g],p[g])
\ee
with $(u[g],w[g],p[g])$ solution of \eqref{FSI-model} is continuously differentiable.
 \end{theorem}
 \begin{remark}
  Here, it is not necessary to assume Hypothesis \ref{hyp1}, \ref{hyp2}, or \ref{hyp4} explicitly, since by Theorem \ref{NSE:existence} the existence of a solution of the FSI problem is in ball of radius $\tilde{r}$ which we can choose arbitrary small if $r>0$ is chosen accordingly sufficiently small. This guarantees implicitly the existence of a solution to the Navier-Stokes equation making Hypothesis \ref{hyp1} redundant as well as a sufficiently small bound on the velocity of the Navier-Stokes equation and the solution of the elasticity system making Hypothesis \ref{hyp4} and so also Hypothesis \ref{hyp2} redundant.
 \end{remark}
\renewenvironment{proof}{{\noindent\emph{Proof of Theorem \ref{thm:main}}}}{}

\begin{proof}{} We follow ideas from \cite{MR3959888}.
 Existence of a solution of the fluid-structure interaction problem follows by Theorem \ref{NSE:existence}. We have $ (u,w,p)=\Pi(g)$ 
 and 
 \be\label{equ:fixed}
 u =S\bigg(f_1,\calf(\cald(\gamma_{\Gammaint} u),g)\bigg)
 \ee
 with $S$ defined in Theorem \ref{thm:Ciarlet} and $\calf$ given in \eqref{operator-F}. 
 Since $(w , p)$ depends continuously differentiable on $(u, g )$ by Lemma \ref{lem:999}, it is sufficient to show differentiability of the mapping $g  \mapsto u$ 
given by the above fix point relation~\eqref{equ:fixed}. We apply the implicit function theorem. We note that
\be
D_2 S\bigg(f_1,\calf(\cald(\gamma_{\Gammaint} u),g)\bigg)\colon W^{1-1/p,p}(\Om_2) \rar W^{2,p}(\Om_2)  
\ee
corresponds to the solution operator for the elasticity problem \eqref{syst:101}, see Theorem~\ref{thm:Ciarlet} and is hence, bounded.
For  
\be
D_u \calf( \cald (\gamma_{\Gammaint} u ) , g )(\delta u) \colon W^{2,p}( \Om_1) \rar W^{1 - 1 / p , p} ( \Gammaint)
\ee
we use that by Lemma \ref{lem:shape} the norm $\norm{D_u\calf}{L_F}$ can be made arbitrarily small choosing $r$ sufficiently small and taking the continuous dependence of the solution of the FSI problem on the data into account, see Theorem \ref{NSE:existence}. Thus,  
$\id - D_2 S \circ D_u \calf$ 
is invertible. By the implicit function theorem we obtain the continous differentiability of the mapping $\Pi$.
\end{proof}

 \renewenvironment{proof}{{\noindent\emph{Proof}}}{}

\appendix

    \section{Transformation of the Navier-Stokes equation}\label{app:A}
    Following~\cite{MR3825153} we state the strong and weak formulation of the Navier-Stokes equation in the physical and reference domain. We have for the velocity $(\wt_1,\wt_2)$ and pressure $\pt$ in the physical domain $\Om_2[u]$ 
 \begin{equation}
    \begin{aligned}\label{euq1}
     -\nu \Delta_x \wt_1 + \wt^{\top} \nabla \wt_1 + (\nabla \pt)_1 &= 0 && \text{in } \Om_2[u],\\ 
     -\nu \Delta_x \wt_2 + \wt^{\top} \nabla \wt_2 + (\nabla \pt)_2 &= 0 && \text{in } \Om_2[u],\\ 
     \dddiv  \wt &=0 &&\text{in } \Om_2[u],\\ 
     \wt &=\delta g &&\text{on }\Gammain ,\\ 
     \wt &=0 &&\text{on }\Gammawall \cup \Gammaint,\\ 
     -\nu D \wt \cdot \ny + \pt \cdot \ny &=0 &&\text{on } \Gammaout.
     \end{aligned}
    \end{equation}
    Transforming to a weak form by multiplying with a test function, integration over $\Om_2[u]$, and apply integration by parts we obtain
    \be
    \begin{aligned}
    -\nu \int_{\Gammaout} \psit_1\nabla \wt_1 \ny \dd s_y + \nu \int_{\Om_2[u]} (\nabla \psit_1 )^{\top} (\nabla \wt_1) \dd y \\
     + \int_{\Om_2[u]} \psit_1 (\wt^{\top} \nabla) \wt_1 \dd y + \int_{\Gammaout} \psit_1 \pt(\ny)_1 \dd s_y \\
     - \int_{\Om_2[u]} \pt (\nabla \psit_1)_1 \dd y=:I_1 + I_2 + I_3 + I_4 + I_5=0.
    \end{aligned}
    \ee
    We have by \eqref{lem:rel-2}, \eqref{lem:rel-3}, and \eqref{outer-normal} on the do-nothing outflow boundary part
    \be
     \begin{aligned}
     I_1 &:=-\nu \int_{\Gammaout} \psit_1\nabla \wt_1 \ny \dd s_y\\
         & = - \nu \int_{\Gammaout} \psi_1(G[u]^{-1} \nabla w_1)^{\top} \frac{K[u] \nx}{\norm{K[u]\nx}{}} \norm{K[u]\nx}{} \dd s_x \\
         &=  - \nu \int_{\Gammaout} \psi_1(\nabla w_1)^{\top} \left(\frac{1}{J} K^{\top} K \right)\nx \dd s_x\\
         &=  - \nu \int_{\Gammaout} \psi_1(\nabla w_1)^{\top} A \nx \dd s_x.
    \end{aligned}
    \ee
    For the diffusion term we have using \eqref{lem:rel-3}
    \be
    \begin{aligned}
    I_2 &:= \nu \int_{\Om_2[u]} (\nabla \psit_1 )^{\top} (\nabla \wt_1) \dd y \\
    & = \nu \int_{\Om_2} (\frac{1}{J} K\nabla \psi_1 )^{\top} (\frac{1}{J}K\nabla w_1) J \dd y \\
    &=\nu \int_{\Om_2} (\nabla \psi_1)^{\top} A(\nabla w_1) \dd x \\
    & =\nu \int_{\Gammaout} \psi_1(\nx^{\top} A \nabla w_1) \dd s_x - \nu \int_{\Om_2} \psi_1 \nabla^{\top} (A \nabla w_1) \dd x.
    \end{aligned}
    \ee
  The convection term transforms using \eqref{lem:rel-3} as follows
  \be
    \begin{aligned}
      I_3 &:=  \int_{\Om_2[u]} \psit_1 (\wt^{\top} \nabla) \wt_1 \dd y 
      = \int_{\Om_2} \psi_1 w^{\top} \frac{1}{J} K \nabla w_1 J \dd x = \int_{\Om_2} \psi_1 w^{\top} K \nabla w_1 \dd x.
    \end{aligned}
    \ee
    For the boundary pressure term we have by \eqref{lem:rel-2} and \eqref{outer-normal} 
    \be
    \begin{aligned}
      I_4 & := \int_{\Gammaout} \psit_1 \pt(\ny)_1 \dd s_y\\
      &= \int_{\Gammaout} \psi_1 \left( p \frac{K\nx}{\norm{K \nx}{}} \right)_1 \norm{K\nx}{} \dd s_x \\
      & =\int_{\Gammaout} \psi_1 p (K\nx)_1 \dd s_x.
    \end{aligned}
    \ee
    Finally, for the volume pressure term we have
    \be
    \begin{aligned}
      I_5 &:= - \int_{\Om_2[u]} \pt (\nabla \psit_1)_1 \dd y\\
      &= - \int_{\Om_2} p(K \nabla \psi_1)_1 \dd x \\
      & = - \int_{\Gammaout} \psi_1 p(K\nx)_1 \dd s_x + \int_{\Om_2} \psi_1 p \operatorname{div}_x (Kp)_1 \dd  x,
    \end{aligned}
    \ee
    where 
    \be
    \operatorname{div}_x(Kp)_1 := \partial_{x_1}(k_{11} p) + \partial_{x_2}(k_{12} p).
    \ee
    Summarizing we obtain the weak formulation
    \be
    \begin{aligned}
      &- \nu \int_{\Gammaout } \psi_1(\nabla w_1)^{\top} A \nx \dd s_x +\nu \int_{\Om_2} (\nabla \psi_1)^{\top} A(\nabla w_1) \dd x\\
      & + \int_{\Om_2} \psi_1 w^{\top} K \nabla w_1 \dd x + \int_{\Gammaout } \psi_1 p (K\nx)_1 \dd s_x+ \int_{\Om_2} \psi_1 \operatorname{div}_x (Kp)_1 \dd  x=0
    \end{aligned}
    \ee
    \begin{multline}\label{NS-weak}
      \nu \int_{\Om_2} (\nabla \psi_1)^{\top} A(\nabla w_1) \dd x
       + \int_{\Om_2} \psi_1 w^{\top} K \nabla w_1 \dd x  \\
       +  \int_{\Om_2} \psi_1 \operatorname{div}_x (Kp)_1 \dd  x= \int_{\Gammaout } f_3 v \dd s +
      \int_{\Om_2} f v \dd x
    \end{multline}
    and equivalently in strong form
    \be
\left\{
\begin{aligned}
-\nu \nabla(A[u] \nabla w) + w(K[u]\nabla) w + K[u]\nabla p&=0&&\text{in }\Om_2,\\
\operatorname{div}_{K^{\top}(u)} w&=0&&\text{in }\Om_2,\\
w&= \delta g &&\text{on }\Gammain,\\
w&=0&&\text{on }\Gamma_{\text{wall}} \cup \Gamma_{\text{int}},\\
- \nu \partial_{A[u],n} w + pK[u] \cdot \nx &=0&&\text{on } \Gamma_{\text{out}}.
\end{aligned}
\right.
\ee

 \section{Transformation of the linearized Navier-Stokes equation}
 For the velocity $(\wt_1,\wt_2)$ and pressure $\pt$ in the physical domain $\Om_2[u]$ we have
    \begin{equation}\label{euq1-app}
    \begin{aligned}
     -\nu \Delta_x \zt_{w_1} + \wh^{\top} \nabla \zt_{w_1} + \zt_{w}^{\top} \nabla \wh_1 + (\nabla \zt_{p})_1 &= 0 && \text{in } \Om_2[u],\\ 
     -\nu \Delta_x \zt_{w_2} + \wt^{\top} \nabla \zt_{w_2} + \zt_{w}^{\top} \nabla \wh_2 + (\nabla \zt_{p})_2 &= 0 && \text{in } \Om_2[u],\\ 
     \dddiv \zt_{w} &=0 &&\text{in } \Om_2[u],\\ 
    \zt_w &= \delta g &&\text{on }\Gammain,\\
     \zt_{w} &=0 &&\text{on }\Gammawall \cup \Gammaint ,\\ 
     -\nu \partial_n \zt_{w} + \zt_{p} \cdot \ny &=0 &&\text{on } \Gammaout .
     \end{aligned}
    \end{equation}
   All linear terms are transformed as for the Navier-Stokes equation.
    The first term of the linearized convection term transforms using \eqref{lem:rel-3} as follows
  \be
    \begin{aligned}
      \int_{\Om_2[u]} \psit_1 (\wt^{\top} \nabla) z_{\wt_1} \dd y 
      = \int_{\Om_2} \psi_1 w^{\top} \frac{1}{J} K \nabla z_{w_1} J \dd x = \int_{\Om_2} \psi_1 w^{\top} K \nabla z_{w_1} \dd x
    \end{aligned}
    \ee   
    and the second one accordingly. \if{ as
  \be
    \begin{aligned}
      I_3 &=  \int_{\Om_2[u]} \psit_1 (\wt^{\top} \nabla) \wt_1 \dd y
      = \int_{\Om_2} \psi_1 z_w^{\top} \frac{1}{J} K \nabla w_1 J \dd x = \int_{\Om_2} \psi_1 z_w^{\top} K \nabla w_1 \dd x.
    \end{aligned}
    \ee}\fi
    That means we have for the transformed equation in strong form
     \be
\left\{
\begin{aligned}
-\nu \nabla(A[u] \nabla z_w) + z_w(K[u]\nabla) \wh + \wh (K[u]\nabla) z_w  + K[u]\nabla z_p&=0&&\text{in }\Om_2,\\
\operatorname{div}_{K[u]^{\top}} z_w&=0&&\text{in }\Om_2,\\
z_w&=0&&\text{on }\Gammain,\\
z_w&= \delta g&&\text{on }\Gamma_{\text{wall}} \cup \Gamma_{\text{int}},\\
- \nu \partial_{A[u],\nx } z_w + z_p K[u] \cdot \nx &=0&&\text{on } \Gamma_{\text{out}}.
\end{aligned}
\right.
\ee
    
    \section{Some properties}\label{app:prop}

\begin{lemma}[Algebra property]\label{lem:algebra} Let $\Om\subset \RR^2$ be open and bounded. Furthermore, let $p$ and $q$ be real with $2 <p< \infty$, $p \ge q \ge 1$. Then, for $v \in W^{1,p}(\Om)$ and $u \in W^{1,q}(\Om)$, the product $uv$ belongs
to $W^{1,q}(\Om)$, and we have
\be
\norm{uv}{W^{1,q}(\Om)} \le \Om \norm{u}{W^{1,p}(\Om)} \norm{v}{W^{1,q(\Om)}}.
\ee
\end{lemma}
\begin{proof}
 Immediate.
\end{proof}

    With the embedding of Sobolev in H\"older spaces we have for $p>2$
    \be
    W^{2,p}(\Om_2) \subset C^{1,\beta}(\bar{\Om}_2) \subset C^{0,1}(\bar{\Om}_2)\quad  \text{for some } \beta>0
    \ee
    and so \cite[p. 338 and p. 325]{alt}
    \be\label{est:sobolev}
    \norm{v}{W^{1,\infty}(\Om_2)} = \norm{v}{C^{0,1}(\bar{\Om}_2)} \le c  \norm{v}{W^{2,p}(\Om_2)} \quad \text{for } v \in  W^{2,p}(\Om_2).
    \ee

For $w\in W^{1,2}(\Om_2)^2$ and recalling $K[u]$ we have the following calculus rules: 
    \be
    \begin{aligned}
      \dddiv (K[u]) &=0 \quad \text{(Piola's identity)},\\
         \dddiv_{\id  - K[u]^{\top} } w& =((\id - K[u])\nabla)^{\top} w=\dddiv w  - K[u]^{\top} \cdot \nabla w =  (\id - K[u]^{\top}) \cdot \nabla w.
    \end{aligned}
    \ee

\newcommand{\etalchar}[1]{$^{#1}$}

\end{document}